\documentclass[reqno,11pt]{amsart}
\theoremstyle{plain}
\newtheorem{thm}{Theorem}[section]
\newtheorem{lemma}[thm]{Lemma} %%Delete [thm] to re-start numbering
\newtheorem{prop}[thm]{Proposition}

\theoremstyle{remark}

\theoremstyle{definition}
\newtheorem{defi}[thm]{Definition}

\newtheorem{ques}[thm]{Question}

\usepackage{amscd,amssymb,comment,epic,eepic,euscript,graphics}
\usepackage{enumitem}
\usepackage[initials]{amsrefs}
\usepackage{color}
\def\cblu{\color{blue}}

\newcommand\At{{\widetilde A}}
\newcommand\Cpx{{\mathbf C}}
\newcommand\diag{\text{\rm diag}}
\newcommand\Eb{{\mathbb E}}
\newcommand\Ec{{\mathcal{E}}}
\newcommand\Ect{{\widetilde{\Ec}}}
\newcommand\eps{\epsilon}
\newcommand\Ints{{\mathbf Z}}
\newcommand\lspan{\mathrm{span}\,}
\newcommand\Nats{{\mathbf N}}
\newcommand\oup{^{\mathrm o}}
\newcommand\Pc{{\mathcal{P}}}
\newcommand\PC{\operatorname{PC}}
\newcommand\pit{{\tilde\pi}}
\newcommand\Reals{{\mathbf R}}

\newcommand\tr{{\mathrm{tr}}}

\begin{document}

\title[Random Vandermonde matrices]{Asymptotic $*$-moments of some random Vandermonde matrices}

\author[Boedihardjo]{March Boedihardjo$^*$}
\address{M.\ Boedihardjo, Department of Mathematics, Texas A\&M University, College Station, TX 77843-3368, USA}
\email{march@math.tamu.edu}
\thanks{{}$^*$Research supported in part by NSF grant DMS-1301604.
This work is included in the thesis of M. Boedihardjo for partial fulfillment of the requirements to obtain a Ph.D.\ degree
at Texas A\&M University}

\author[Dykema]{Ken Dykema$^\dag$}
\address{K.\ Dykema, Department of Mathematics, Texas A\&M University, College Station, TX 77843-3368, USA}
\email{kdykema@math.tamu.edu}
\thanks{{}$^\dag$Research supported in part by NSF grant DMS-1202660.}


\date{July 21, 2017}

\begin{abstract}
Appropriately normalized square random Vandermonde matrices based on independent random variables with uniform
distribution on the unit circle are studied.
It is shown that as the matrix sizes increases without bound,
with respect to the expectation of the trace
there is an asymptotic $*$-distribution, equal to that of a $C[0,1]$-valued R-diagonal element.
\end{abstract}

\maketitle

\section{Introduction}

We consider the random Vandermonde matrix $X_N$, whose $(i,j)$-th entry is $N^{-1/2}\zeta_i^j$,
where $\zeta_1,\ldots,\zeta_N$ are independent with Haar measure distribution on the unit circle.
These have been studied in~\cite{RyanDebbah}, \cite{RyanDebbah11}, \cite{TW11} and~\cite{TW14}
and are of interest for applications in finance, signal array processing, wireless communications
and biology (see~\cite{RyanDebbah} for references).
In~\cite{RyanDebbah}, Ryan and Debbah show that asymptotic moments of $X_{N}^{*}X_{N}$,
(namely, the limits
\[
\lim_{N\to\infty}\Eb\circ\tr((X_N^*X_N)^p),
\]
where $\Eb$ is the expectation and $\tr$ is the normalized trace on matrix algebras),
exist and are given by sums of volumes of certain polytopes.
They also compute some of these asymptotic moments.
In~\cite{TW11}, Tucci and Whiting show among other things that the
asymptotic moments are given by
\[\lim_{N\to\infty}\mathbb{E}\circ\mathrm{tr}((X_{N}^{*}X_{N})^{p})=\int x^{p}\,d\mu(x)\]
for a unique measure $\mu$ on $[0,\infty)$ with unbounded support.
(This uses the Stieltjes solution to the moment problem and a theorem of Carleman --- for the former,
see p.\ 76 of~\cite{Akh}.)
Further results are proved in~\cite{RyanDebbah11} and~\cite{TW14}.

G.\ Tucci asked~\cite{Tucci} whether $X_N$ is asymptotically R-diagonal with respect
to the expectation of the trace.
In this paper, we answer Tucci's question negatively, but show that $X_N$ has an asymptotic $*$-distribution
as $N\to\infty$, which is in fact the $*$-distribution of an element that is R-diagonal over the C$^*$-algebra $C[0,1]$.

To be precise, we show that, for all $n\in\Nats$ and all $\eps(1),\ldots,\eps(n)\in\{1,*\}$,
\[
\lim_{N\to\infty}\Eb\circ\tr(X_N^{\eps(1)}\cdots X_N^{\eps(n)})
\]
exists and we describe this limiting
$*$-moment using the notion of $C[0,1]$-valued R-diagonal\-ity.

Usual (or scalar-valued) R-diagonal elements are very natural in free probability theory,
and have been much studied; they were introduced by Nica and Speicher in~\cite{NS97}.
The algebra-valued version was introduced by {\'S}niady and Speicher in~\cite{SS01}
and has been further studied in~\cite{BD:rdiag}.
We will give the definition from~\cite{BD:rdiag}, which is an easy reformulation of one of the characterizations in~\cite{SS01}.

The setting for algebra-valued R-diagonal elements is a $B$-valued $*$-noncommutative probability space $(A,\Ec)$,
where $B\subseteq A$ is a unital inclusion of unital $*$-algebras and $\Ec:A\to B$ is a conditional expectation,
namely, a $B$-bimodular unital projection.

\begin{defi}\label{def:maxaltptn}
Given $n\in\Nats$ and $\eps=(\eps(1),\ldots,\eps(n))\in\{1,*\}^n$,
we define the {\em maximal alternating interval partition} $\sigma(\eps)$ to be the interval partition
of $\{1,\ldots,n\}$ whose blocks are the maximal interval subsets $I$ of $\{1,\ldots,n\}$ such that if $j\in I$ and $j+1\in I$,
then $\eps(j)\ne\eps(j+1)$.
\end{defi}

For example, if $\eps=\{1,1,*,1,*,*\}$, then $\sigma(\eps)=\{\{1\},\{2,3,4,5\},\{6\}\}$.

\begin{defi}\label{def:rdiag}
An element $a\in A$ is {\em $B$-valued R-diagonal}
if for every integer $k\ge0$ and every $b_1,\ldots,b_{2k}\in B$ we have
\[
\Ec(ab_1a^*b_2ab_3a^*\cdots b_{2k-2}ab_{2k-1}a^*b_{2k}a)=0,
\]
(namely, odd alternating moments vanish)
and, for every integer $n\ge1$, every $\eps\in\{1,*\}^n$ and every choice of $b_1,b_2,\ldots b_n\in B$, we have
\[
\Ec\left(\prod_{I\in\sigma(\eps)}\left(\left(\prod_{j\in I}b_ja^{\eps(j)}\right)-\Ec\left(\prod_{j\in I}b_ja^{\eps(j)}\right)\right)\right)=0,
\]
where in each of the three products above, the terms are taken in order of increasing indices.
\end{defi}

Note that the $B$-valued R-diagonality condition determines all of the $B$-valued $*$-moments
\[
\Ec\big(a^{\eps(1)}b_1a^{\eps(2)}\cdots b_{n-1}a^{\eps(n)}\big)
\]
for $n\in\Nats$, $b_1,\ldots,b_{n-1}\in B$ and arbitrary $\eps(1),\ldots,\eps(n)\in\{1,*\}$,
in terms of the alternating moments of even length, namely those when $n$ is even and $\eps(j)\ne\eps(j+1)$ for all $j$.

\smallskip\noindent
{\bf Contents:}
The contents of the rest of the paper are as follows.
In Section~\ref{sec:altmoms}, we find asymptotics of diagonal entries of
$*$-moments involving alternating $X_N$ and $X_N^*$
with certain deterministic diagonal matrices between.
In Section~\ref{sec:Rdiag}, we prove our main result, characterizing arbitrary asymptotic $*$-moments of $X_N$
based on $C[0,1]$-valued R-diagonality.
In Section~\ref{sec:calcs}, we prove results allowing the asymptotic  alternating $*$-moments of $X_N$
found in Section~\ref{sec:altmoms} to be computed in terms of certain integrals, we show that $X_N$ is not
asymptotically scalar-valued R-diagonal, and we report the results of computations of certain $C[0,1]$-valued cumulants
of the asymptotic $*$-distribution of $X_N$.
(Details of these computations can be found in a Mathematica~\cite{W15} file accompanying the arXived version of this paper.)

\smallskip\noindent
{\bf Notation:}
On matrix algebras, $\mathrm{tr}$ is the normalized trace and $\mathrm{Tr}$ is the usual trace.
For partitions $\pi_1$ and $\pi_2$ of the same set, $\pi_1\vee\pi_2$ means their join in the lattice of all partitions of the set.
We say that a set $S$ {\em splits} a partition $\pi$ if $S$ is the union of some of the blocks of $\pi$.
We write $k_{1}\stackrel{\pi}{\sim}k_{2}$ to mean that $k_{1}$ and $k_{2}$ are in the same block of $\pi$.
The {\em restriction} of a partition $\pi$ to a set $K$ is
the partition $\{S\cap K:S\in\pi\}\backslash\{\emptyset\}$,
and is denoted $\pi\upharpoonright_{K}$.
If $i$ is a function with domain $L$, then $\ker i$ is the partition of $L$ so that $\ell_1$ and $\ell_2$ belong to the same
block of $\ker i$ if and only if $i(\ell_1)=i(\ell_2)$.

\smallskip\noindent
{\bf Acknowledgement:} The authors are grateful to Gabriel Tucci for helpful discussions about random
Vandermonde matrices at an early stage of this research.

\section{Asymptotic alternating $C[0,1]$-valued $*$-moments.}
\label{sec:altmoms}

In this section, we investigate alternating moments in $X_N$ and $X_N^*$.
More specifically, we find the asymptotics of the expectations of diagonal elements of alternating moments of even length,
with certain non-random diagonal matrices interspersed (see Proposition~\ref{prop:diagexpect}).

Let $\tau$ be the tracial state on $C[0,1]$ given by integration with Lebesgue measure.

Given $n\in\Nats$, we let $\Pc(n)$ denote the lattice of all set partitions of $\{1,2,\ldots,n\}$.
Thus $\pi\in\Pc(n)$ if and only if $\pi$ is a collection of disjoint, nonempty sets whose union is $\{1,\ldots,n\}$.
As usual, the elements of $\pi$ are called blocks of the partition, and $|\pi|$ is the number of blocks in the partition.
We will let $S_\pi(j)$ denote the block of $\pi$ that has $j$ as an element.

For $\pi\in\Pc(n)$ and $g_1,\ldots g_n\in C[0,1]$, let $\Gamma_\pi(g_1,\ldots,g_n)\in C[0,1]$ be defined by
\[
\Gamma_\pi(g_1,\ldots,g_n)=
\prod_{j\in S_\pi(1)}g_j\prod_{S\in\pi\backslash\{S_\pi(1)\}}\tau\Big(\prod_{j\in S}g_j\Big)
\]
Thus,
\[
\tau\big(\Gamma_\pi(g_1,\ldots,g_n)\big)=\prod_{S\in\pi}\tau\Big(\prod_{j\in S}g_j\Big)
\]

Given $S\in\pi$, we let $S'=S\backslash\{\max(S)\}$ be $S$ without its largest element
and we let
\begin{equation}\label{eq:J}
J_\pi=\bigcup_{S\in\pi}S'.
\end{equation}
Thus $|J_\pi|=n-|\pi|$.
Naturally, we write $S_\pi'(j)$ for $(S_\pi(j))'$.
For $p\in\{1,\ldots,n\}$ and $S\in\pi$, we write $S\le p$ if and only if  $j\le p$ for every $j\in S$, and if this is not the case, then we
write $S\not\le p$.
We set
\[
I_\pi(p)=\{j\in\{1,\ldots,p\}:S_\pi(j)\not\le p\}
\]
and note $I_\pi(p)\subseteq J_\pi$.

If $ J_\pi=\emptyset$, namely, if $\pi=0_n$ is the partition of $\{1,\ldots,n\}$ into singletons, then we let
\[
\Lambda_\pi(g_1,\ldots,g_{n-1})
=\prod_{j=1}^{n-1}g_j,
\]
where if $n=1$ then we let $\Lambda_\pi()=1$ be the constant function $1$.
Otherwise, if $J_\pi\ne\emptyset$, then
for $t\in[0,1]$ we let
\[
E(\pi,t)=\left\{(t_j)_{j\in J_\pi}\in\Reals^{J_\pi}\;\Bigg|\;\forall p\in\{1,\ldots,n-1\},\,0<t+\sum_{j\in I_\pi(p)}t_j\le 1\right\}.
\]
and we set
\begin{equation}\label{eq:Lambdapi}
\Lambda_\pi(g_1,\ldots,g_{n-1})(t)
=\underset{E(\pi,t)}\int\left(\prod_{p=1}^{n-1}g_p\Big(t+\sum_{j\in I_\pi(p)}t_j\Big)\right)\,d\lambda((t_j)_{j\in J_\pi}),
\end{equation}
where the integral is with respect to $|J_\pi|$-dimensional Lebesgue measure.

The following lemma provides an alternative description of $\Lambda_\pi$ that may be more natural.
It will be used in the proof of Proposition~\ref{prop:diagexpect} and in Section~\ref{sec:calcs}.
\begin{lemma}\label{lem:LambdapiAlt}
Assume $\pi\in\Pc(n)\backslash\{0_n\}$
and let $\Phi_\pi:\Reals^{J_\pi}\times\Reals\to\Reals^n$ be the linear mapping given by
\[
\Phi_\pi((t_j)_{j\in J_\pi},t)=\bigg(t+\sum_{j\in I_\pi(p)}t_j\bigg)_{p=1.}^n
\]
Then $\Phi_\pi$ is an isomorphism onto
the subspace
\begin{equation}\label{eq:Kpi}
K_\pi= \bigg\{(s_p)_{p=1}^n\bigg|\;\forall S\in\pi,\,\sum_{p\in S}s_p-s_{p-1}=0\bigg\}
\end{equation}
of $\Reals^n$, using the convention $s_0=s_n$.
Furthermore, we have
\begin{equation}\label{eq:KZ}
\Phi_\pi(\Ints^{J_\pi}\times\Ints)=K_\pi\cap\Ints^n.
\end{equation}

For each $t\in[0,1]$, let $\Phi_\pi^{(t)}:\Reals^{J_\pi}\to\Reals^n$ be
$\Phi_\pi^{(t)}(\cdot)=\Phi(\cdot,t)$.
Then $\Phi_\pi^{(t)}$ maps $E(\pi,t)$ onto
\[
F(\pi,t):=\{(s_p)_{p=1}^n\in K_\pi\cap(0,1]^n:s_n=t\}.
\]
Moreover, letting ${\cblu\nu_\pi^{(t)}}$  be the push-forward of Lebesgue measure on $\Reals^{J_\pi}$  by $\Phi_\pi^{(t)}$, we have
\begin{equation}\label{eq:intF}
\Lambda_\pi(g_1,\ldots,g_{n-1})(t)=\int_{F(\pi,t)}(g_1\otimes\cdots\otimes g_{n-1}\otimes 1)\,d\nu_\pi^{(t)}.
\end{equation}
\end{lemma}
\begin{proof}
Let $(s_p)_{p=1}^n=\Phi_\pi((t_j)_{j\in J_\pi},t)$.
Since $I_\pi(n)=\emptyset$, we have $s_n=t$.
We will show
\begin{equation}\label{eq:sptp}
\forall p\in\{1,\ldots,n\}\quad
s_p-s_{p-1}=\begin{cases}
t_p,&p\ne\max S_\pi(p) \\
-\sum_{i\in S_\pi(p)\backslash\{p\}}t_i,&p=\max S_\pi(p).
\end{cases}
\end{equation}
Suppose $p\ne\max S_\pi(p)$.
If $p=1$, then $I_\pi(p)=\{1\}$ and 
\[
s_p-s_{p-1}=s_1-s_n=(t+t_1)-t=t_1.
\]
If $p>1$, then $I_\pi(p)=I_\pi(p-1)\cup\{p\}$ and $s_p-s_{p-1}=t_p$.

On the other hand, suppose $p=\max S_\pi(p)$.
If $p=1$, then $S_\pi(p)=\{1\}$ and $I_\pi(p)=\emptyset$ and
\[
s_p-s_{p-1}=s_1-s_n=t-t=0=-\sum_{j\in S_\pi(p)\backslash\{p\}}t_j.
\]
If $p>1$, then
\[
I_\pi(p-1)=I_\pi(p)\cup(S_\pi(p)\backslash\{p\})
\]
and
\[
s_p-s_{p-1}=-\sum_{i\in S_\pi(p)\backslash\{p\}}t_i.
\]
This proves~\eqref{eq:sptp}.

From~\eqref{eq:sptp}, we see immediately $\Phi_\pi(\Reals^{J_\pi}\times\Reals)\subseteq K_\pi$.

We will now show that $\Phi_\pi$ is injective.
Indeed, if $\Phi_\pi((t_j)_{j\in J_\pi})=(s_p)_{p=1}^n=(0)_{p=1}^n$, then $t=s_n=0$ and, for all $p\in J_\pi$, using~\eqref{eq:sptp}, we have $t_p=s_p-s_{p-1}=0$.

Note that $K_\pi$ is the solution space of $|\pi|$ linear equations, but that the sum of all of these linear equations is $0$, so that the dimension
of $K_\pi$ is at most $n-|\pi|+1$.
But, since $\Phi_\pi$ is an injective linear transformation into $K_\pi$, the dimension of $K_\pi$ is at least $|J_\pi|+1=n-|\pi|+1$.
Thus, $\Phi_\pi$ is an isomorphism.

In order to prove~\eqref{eq:KZ}, note that the inclusion $\subseteq$ follows immediately from the definition of $\Phi_\pi$.
The reverse inclusion holds because if $\Phi_\pi((t_j)_{j\in J_\pi},t)=(s_p)_{p=1}^n\in\Ints^n$, then $t=s_n\in\Nats$, while for every $j\in J_\pi$,
by~\eqref{eq:sptp}, $t_j=s_j-s_{j-1}\in\Ints$.

It is now clear that $\Phi_\pi^{(t)}$ maps $E(\pi,t)$ onto $F(\pi,t)$.
It remains only to prove~\eqref{eq:intF}.
From the definition~\eqref{eq:Lambdapi} and the definition of $\Phi_\pi^{(t)}$, we see
\[
\Lambda_\pi(g_1,\ldots,g_{n-1})(t)=\int_{E(\pi,t)}\big((g_1\otimes\cdots\otimes g_{n-1}\otimes 1)\circ\Phi_\pi^{(t)}\big)((t_j)_{j\in J_\pi})\,d\lambda((t_j)_{j\in J_\pi}).
\]
This is, of course, equal to the integral on the right hand side of~\eqref{eq:intF}, by the definition of the push-forward measure.
\end{proof}

\begin{prop}\label{prop:diagexpect}
Let $n\in\Nats$ and suppose $g_1,\ldots,g_{2n}\in C[0,1]$.
Given $N\in\Nats$ and $j\in\{1,\ldots,2n\}$ consider the deterministic $N\times N$ diagonal matrix 
\[
D_N^{(j)}=\diag({\textstyle g_j(\frac 1N),g_j(\frac 2N),\ldots,g_j(\frac NN)}).
\]
For $t\in[0,1]$, let $h_N(t)$ be the least element of $\{1,\ldots,N\}$ so that $t\le h_N(t)/N$.
Then for all $t\in[0,1]$
\begin{multline}\label{eq:diagexpect1}
\lim_{N\to\infty}
\Eb(D_N^{(1)}X_N^*D_N^{(2)}X_ND_N^{(3)}X_N^*D_N^{(4)}X_N\cdots D_N^{(2n-1)}X_N^*D_N^{(2n)}X_N)_{h_N(t),h_N(t)} \\
=\sum_{\pi\in\Pc(n)}g_1(t)\Lambda_\pi(g_3,g_5,\ldots,g_{2n-1})(t)\tau\big(\Gamma_\pi(g_2,g_4,\ldots,g_{2n})\big)
\end{multline}
and
\begin{multline}\label{eq:diagexpect2}
\lim_{N\to\infty}
\Eb(D_N^{(1)}X_ND_N^{(2)}X_N^*D_N^{(3)}X_ND_N^{(4)}X_N^*\cdots D_N^{(2n-1)}X_ND_N^{(2n)}X_N^*)_{h_N(t),h_N(t)} \\
=\sum_{\pi\in\Pc(n)}\Gamma_\pi(g_1,g_3,\ldots,g_{2n-1})(t)\tau\big(\Lambda_\pi(g_2,g_4,\ldots,g_{2n-2})g_{2n}\big)
\end{multline}
Furthermore, in both cases the convergence is uniform for $t\in[0,1]$, and the rate of convergence can be
controlled
in terms of only $\max_i\|g_i\|$ and a common modulus of continuity for $\{g_1,\ldots,g_{2n}\}$.
\end{prop}
\begin{proof}
We have
\begin{multline}\label{eq:ExpX*DXhh}
\Eb(D_N^{(1)}X_N^*D_N^{(2)}X_ND_N^{(3)}X_N^*D_N^{(4)}X_N\cdots D_N^{(2n-1)}X_N^*D_N^{(2n)}X_N)_{h_N(t),h_N(t)} \\
=\begin{aligned}[t]
N^{-n}&\sum_{\{(i(1),\ldots,i(2n))\in\{1,\ldots,N\}^{2n}:i(1)=h_N(t)\}}\left(\prod_{j=1}^{2n}g_j(\frac{i(j)}N)\right) \\ 
 &\qquad\qquad\qquad\cdot\Eb(\zeta_{i(2)}^{-i(1)+i(3)}\zeta_{i(4)}^{-i(3)+i(5)}\cdots\zeta_{i(2n-2)}^{-i(2n-3)+i(2n-1)}\zeta_{i(2n)}^{-i(2n-1)+i(1)}).
\end{aligned}
\end{multline}
Let us rearrange the sum by summing first over all partitions $\pi\in\Pc(n)$ and then over all
$i_e=(i(2),i(4),\ldots,i(2n))\in\{1,\ldots,N\}^{n}$ such that $\ker i_e=\pi$, and then over all
$i_o=(i(1),i(3),\ldots,i(2n-1))\in\{1,\ldots,N\}^{n}$ such that $i(1)=h_N(t)$,
where $\ker i_e=\pi$ means that $i(2j_1)=i(2j_2)$ if and only if $j_1$ and $j_2$ are in the same block of $\pi$.
Keeping in mind that the $\zeta_j$ are independent and $\Eb(\zeta_j^k)=0$ if $k\ne0$, we find
that the expectation in~\eqref{eq:ExpX*DXhh} equals
\begin{equation}\label{eq:twosums}
N^{-n}\sum_{\pi\in\Pc(n)}\left(\sum_{\ell}\prod_{S\in\pi}\left(\prod_{j\in S}g_{2j}(\frac{\ell_S}N)\right)\right)
\sum_{i_o\in\Psi_1(\pi,N,h_N(t))}\;\prod_{j=1}^ng_{2j-1}(\frac{i(2j-1)}N),
\end{equation}
where the summation $\sum_{\ell}$
is over all $\ell=(\ell_S)_{S\in\pi}\in\{1,\ldots,N\}^{|\pi|}$ such that $\ell_{S_1}\ne\ell_{S_2}$ if $S_1\ne S_2$,
and $\Psi_1(\pi,N,h_N(t))$ is the set of all $i_o=(i(1),i(3),\ldots,i(2n-1))\in\{1,\ldots,N\}^n$ such that $i(1)=h_N(t)$ and
for all $S\in\pi$,
\[
\sum_{j\in S}-i(2j-1)+i(2j+1)=0,
\]
with the convention $i(2n+1)=i(1)$.
It is straightforward from the theory of Riemann integration to see
\[
\lim_{N\to\infty}N^{-|\pi|}\left(\sum_\ell\prod_{S\in\pi}\left(\prod_{j\in S}g_{2j}(\frac{\ell_S}N)\right)\right)
=\prod_{S\in\pi}\tau\left(\prod_{j\in S}g_{2j}\right)
\]
and that the rate of convergence depends only on $\max_j\|g_{2j}\|$ and on a common modulus of continuity of $\{g_2,g_4,\ldots,g_{2n}\}$

Now we analyze the last summation in~\eqref{eq:twosums}.
Let $\Psi_1\oup(\pi,N,h_N(t))$ be obtained by left rotating each element of $\Psi_1(\pi,N,h_N(t))\}$, i.e., 
\[
\text{from}\quad(i(1),i(3),\ldots,i(2n-1))\quad\text{to}\quad(i(3),i(5),\ldots,i(2n-1),i(1)).
\]
Then, in the notation of Lemma~\ref{lem:LambdapiAlt},
\[
\Psi_1\oup(\pi,N,h_N(t))=K_\pi\cap\{1,\ldots,N\}^n.
\]
Thus, using~\eqref{eq:KZ} and the definition of $\Phi_\pi$, we have
\begin{multline*}
\Phi_\pi^{-1}\big(\Psi_1\oup(\pi,N,h_N(t))\big) \\
=\bigg\{((\ell_j)_{j\in J_\pi},h_N(t))\in\Ints^{J_\pi}\times\{h_N(t)\}\;\bigg|\;
\forall p\in\{1,\ldots,n-1\},\;1\le h_N(t)+\sum_{j\in I_\pi(p)}\ell_j\le N\bigg\}
\end{multline*}
and
\begin{align*}
\Phi_\pi^{-1}\bigg(\frac1N\Psi_1\oup(\pi,N,h_N(t))\bigg)&=\frac1N\Phi_\pi^{-1}\big(\Psi_1\oup(\pi,N,h_N(t))\big) \\
&=\bigg(E\bigg(\pi,\frac{h_N(t)}N\bigg)\cap\frac1N(\Ints^{J_\pi})\bigg)\times\bigg\{\frac{h_N(t)}N\bigg\}.
\end{align*}
Thus, we have
\begin{multline}\label{eq:Psi1sum}
N^{-n+|\pi|}\sum_{(i(1),i(3),\ldots,i(2n-1))\in\Psi_1(\pi,N,h_N(t))}\;\prod_{p=1}^ng_{2p-1}\bigg(\frac{i(2p-1)}N\bigg) \\
=N^{-n+|\pi|}g_1\bigg(\frac{h_N(t)}N\bigg)\sum_{(r_j)_{j\in J_\pi}}\;\prod_{p=2}^ng_{2p-1}\bigg(\frac{h_N(t)}N+\sum_{j\in I_\pi(p)}r_j\bigg),
\end{multline}
where the sum is over all $(r_j)_{j\in J_\pi}\in E(\pi,\frac{h_N(t)}N)\cap\frac1N(\Ints^{J_\pi})$.
Since the $g_j$ are continuous, the right hand side of~\eqref{eq:Psi1sum} is for large $N$ a good approximation for the integral
\[
g_1\bigg(\frac{h_N(t)}N\bigg)\int_{E(\pi,\frac{h_N(t)}N)}\left(\prod_{p=2}^ng_{2p-1}\bigg(\frac{h_N(t)}N+\sum_{j\in I_\pi(p)}t_j\bigg)\right)
\,d\lambda((t_j)_{j\in J_\pi}).
\]
In particular, since also $\lim_{N\to\infty}\frac{h_N(t)}N=t$, we have
\begin{multline*}
\lim_{N\to\infty}N^{-n+|\pi|}\sum_{(i(1),i(3),\ldots,i(2n-1))\in\Psi_1(\pi,N,h_N(t))}\;\prod_{p=1}^ng_{2p-1}\bigg(\frac{i(2p-1)}N\bigg) \\
=g_1(t)\Lambda_\pi(g_3,g_5,\ldots,g_{2n-1}),
\end{multline*}
with the rate of convergence depending only on $\max(\|g_{2j-1}\|)$ and a common modulus
of continuity for $\{g_1,g_3,\ldots,g_{2n-1}\}$.
This proves~\eqref{eq:diagexpect1}, with the desired statement on the rate of convergence.

We prove~\eqref{eq:diagexpect2} similarly.
We have
\begin{multline*}
\Eb(D_N^{(1)}X_ND_N^{(2)}X_N^*D_N^{(3)}X_ND_N^{(4)}X_N^*\cdots D_N^{(2n-1)}X_ND_N^{(2n)}X_N^*)_{h_N(t),h_N(t)} \\
=\begin{aligned}[t]
N^{-n}&\sum_{\{(i(1),\ldots,i(2n))\in\{1,\ldots,N\}^{2n}:i(1)=h_N(t)\}}\left(\prod_{j=1}^{2n}g_j(\frac{i(j)}N)\right) \\ 
 &\qquad\qquad\qquad\cdot\Eb(\zeta_{i(1)}^{-i(2n)+i(2)}\zeta_{i(3)}^{-i(2)+i(4)}\cdots\zeta_{i(2n-1)}^{-i(2n-2)+i(2n)}). 
\end{aligned}
\end{multline*}
The right-hand-side can be rewritten
\[
N^{-n}\sum_{\pi\in\Pc(n)}\left(\sum_\ell\prod_{S\in\pi}\left(\prod_{j\in S}g_{2j-1}(\frac{\ell_S}N)\right)\right)
\sum_{i_e\in\Psi_2(\pi,N)}\;\prod_{j=1}^ng_{2j}(\frac{i(2j)}N),
\]
where
the summation $\sum_{\ell}$
is over all $\ell=(\ell_S)_{S\in\pi}\in\{1,\ldots,N\}^{|\pi|}$ such that $\ell_{S_\pi(1)}=h_N(t)$ and $\ell_{S_1}\ne\ell_{S_2}$ if $S_1\ne S_2$,
while
\[
\Psi_2(\pi,N)
=\bigg\{i_e=(i(2),i(4),\ldots,i(2n))\in\{1,\ldots,N\}^n\;\bigg|\;
\forall S\in\pi,\;\sum_{j\in S}i(2j)-i(2j-2)=0\bigg\},
\]
with the convention $i(0)=i(2n)$.
We see
\[
\lim_{N\to\infty}N^{-|\pi|+1}\sum_\ell\prod_{S\in\pi}\left(\prod_{j\in S}g_{2j-1}(\frac{\ell_S}N)\right)
=\Gamma_\pi(g_1,g_3,\ldots,g_{2n-1})(t).
\]

We have $\Psi_2(\pi,N)=K_\pi\cap\{1,\ldots,N\}^n$ and
\[ %begin{multline*}
\Phi_\pi^{-1}(\Psi_2(\pi,N))=
\bigg\{((\ell_j)_{j\in J_\pi},m)\in\Ints^{J_\pi}\times\Ints\;\bigg|\;
\forall p\in\{1,\ldots,n-1\},\;1\le m+\sum_{j\in I_\pi(p)}\ell_j\le N\bigg\}
\] %end{multline*}
and
\begin{equation}\label{eq:PhiInvPsi2}
\Phi_\pi^{-1}\bigg(\frac1N\Psi_2(\pi,N)\bigg)=\frac1N\Phi_\pi^{-1}\big(\Psi_2(\pi,N)\big)
=\bigcup_{m=1}^N\bigg(E\bigg(\pi,\frac mN\bigg)\cap\frac1N(\Ints^{J_\pi})\bigg)\times\bigg\{\frac mN\bigg\}.
\end{equation}
Thus, we have
\[ % \begin{multline}\label{eq:Psi1sum}
\sum_{(i(2),i(4),\ldots,i(2n))\in\Psi_2(\pi,N)}\;\prod_{p=1}^ng_{2p}\bigg(\frac{i(2p)}N\bigg)
=\sum_{((r_j)_{j\in J_\pi},x)}\;\prod_{p=1}^ng_{2p}\bigg(x+\sum_{j\in I_\pi(p)}r_j\bigg),
\]% \end{multline}
where the sum is over all $((r_j)_{j\in J_\pi},x)$ in the set on the right of~\eqref{eq:PhiInvPsi2}.
Thus, using $I_n(\pi)=\emptyset$, we find
\begin{multline*}
\lim_{N\to\infty}N^{-n+|\pi|-1}\sum_{(i(2),i(4),\ldots,i(2n))\in\Psi_2(\pi,N)}\;\prod_{p=1}^ng_{2p}\left(\frac{i(2p)}N\right) \\
\begin{aligned}
&=\int_0^1\bigg(\int_{E(\pi,s)}\bigg(\prod_{p=1}^{n-1}g_{2p}\bigg(s+\sum_{j\in I_\pi(p)}t_j\bigg)\bigg)\,d\lambda((t_j)_{j\in J_\pi})\bigg)\,ds \\
&=\int_0^1\Lambda_\pi(g_2,g_4,\ldots,g_{2n-2})(s)g_{2n}(s)\,ds
\end{aligned}
\end{multline*}
and that the rate of convergence depends only on $\max_{1\le j\le n}\|g_{2j}\|$ and on a common modulus of continuity of $\{g_2,g_4,\ldots,g_{2n}\}$.
This proves~\eqref{eq:diagexpect2}, with the desired statement on the rate of convergence.
\end{proof}

\section{$C[0,1]$-valued R-diagonality}
\label{sec:Rdiag}

In this section, we prove our main theorem (Theorem~\ref{thm:XRDiag}) about asymptotic $*$-moments of random Vandermonde matrices.
It will follow from Proposition~\ref{prop:diagexpect} above, about alternating moments, and the next proposition.
\begin{prop}\label{OneProp}
Let $n\geq 1$. Let $\epsilon_{1},\ldots,\epsilon_{n}\in\{1,*\}$. Let $\sigma\in\Pc(n)$ be the
corresponding maximal alternating interval partition
(see Definition~\ref{def:maxaltptn}).
Let $d_{1},\ldots,d_{n}$ be deterministic diagonal $N\times N$ matrices of norm at most 1. Then
\[\left|\mathbb{E}\circ\mathrm{tr}\prod_{I\in\sigma}\left(\prod_{k\in I}d_{k}X_{N}^{\epsilon_{k}}-\mathbb{E}\circ\mathrm{diag}\left(\prod_{k\in I}d_{k}X_{N}^{\epsilon_{k}}\right)\right)\right|\leq\frac{C}{\sqrt{N}},\]
where $C$ depends only on $n$.
\end{prop}

We begin the proof with some preliminaries.
The following lemma can be proved using Gaussian elimination, for instance.
\begin{lemma}\label{P3}
Let $p\geq 1$. Let $V$ be a subspace of $\mathbb{R}^{p}$. Let $t\in\mathbb{R}^{p}$. Then
\[|\{j\in\{1,\ldots,N\}^{\{1,\ldots,p\}}\cap(t+V)|\leq N^{\dim V}.\]
\end{lemma}
Lemma \ref{P3} can be reformulated follows.
\begin{lemma}\label{P4}
Let $p,r\geq 1$. Let $w_{1},\ldots,w_{r}\in\mathbb{R}^{p}$. Let $m_{1},\ldots,m_{r}\in\mathbb{R}$. Then
\[|\{j\in\{1,\ldots,N\}^{\{1,\ldots,p\}}:j\cdot w_{s}=m_{s}\;\forall 1\leq s\leq r\}|\leq N^{p-\dim\lspan\{w_{1},\ldots,w_{r}\}}.\]
\end{lemma}
\begin{lemma}\label{P1}
Let $\zeta_{1},\ldots,\zeta_{N}$ be independent random variables uniformly distributed on the unit circle. Let $h$ be a product of the random variables $\zeta_{1},\ldots,\zeta_{N}$ and their inverses, possibly with repetitions. Then
\[\mathbb{E}h=\begin{cases}\begin{array}{cc}1,&h=1\\0,&h\neq 1\end{array}.\end{cases}\]
\end{lemma}
\begin{proof}
Obviously if $h=1$ then $\mathbb{E}h=1$. If $h\neq 1$ then we write $h=\prod_{i=1}^{N}\zeta_{i}^{j(i)}$ where $j(i_{0})\neq 0$ for some $1\leq i_{0}\leq N$. Thus, by independence of $\zeta_{1},\ldots,\zeta_{N}$, we have $\mathbb{E}h=(\mathbb{E}\prod_{i\neq i_{0}}\zeta_{i}^{j(i)})(\mathbb{E}\zeta_{i_{0}}^{j(i_{0})})=0$.
\end{proof}
\begin{lemma}\label{P2}
Let $\zeta_{1},\ldots,\zeta_{N}$ be independent random variables uniformly distributed on the unit circle. Let $h_{1},\ldots,h_{r}$ be products of the random variables $\zeta_{1},\ldots,\zeta_{N}$ and their inverses, possibly with repetition. Then
\[|\mathbb{E}(h_{1}-\mathbb{E}h_{1})\cdots(h_{r}-\mathbb{E}h_{r})|\leq\mathbb{E}h_{1}\cdots h_{r}.\]
\end{lemma}
\begin{proof}
If $h_{i}=1$ for some $1\leq i\leq r$ then
\[|\mathbb{E}(h_{1}-\mathbb{E}h_{1})\cdots(h_{r}-\mathbb{E}h_{r})|=0.\]
If $h_{i}\neq 1$ for all $1\leq i\leq r$ then by Lemma \ref{P1}, $\mathbb{E}h_{i}=0$ for all $1\leq i\leq r$ so
\[|\mathbb{E}(h_{1}-\mathbb{E}h_{1})\cdots(h_{r}-\mathbb{E}h_{r})|=|\mathbb{E}h_{1}\cdots h_{r}|=\mathbb{E}h_{1}\cdots h_{r}.\]
\end{proof}
\begin{lemma}\label{P5}
Let $\zeta_{1},\ldots,\zeta_{N}$ be independent random variables uniformly distributed on the unit circle. Let $h$ be a product of the random variables $\zeta_{1},\ldots,\zeta_{N}$ and their inverses, possibly with repetition.
Let $r\geq 1$. Let $i(1),\ldots,i(r)\in\{1,\ldots,N\}$ be distinct. Then there exists $m_{1},\ldots,m_{r}\in\mathbb{Z}$ such that if $n_{1},\ldots,n_{r}\in\mathbb{Z}$ satisfies
\[\mathbb{E}h\zeta_{i(1)}^{n_{1}}\cdots\zeta_{i(r)}^{n_{r}}\neq 0\]
then $n_{s}=m_{s}$ for all $1\leq s\leq r$.
\end{lemma}
\begin{proof}
We write $h$ as $\prod_{i=1}^{N}\zeta_{i}^{j(i)}$. Then the result follows from Lemma~\ref{P1},  by taking $m_{s}=-j(i(s))$ for $1\leq s\leq r$.
\end{proof}
Combining Lemma \ref{P4} and Lemma \ref{P5}, we obtain
\begin{lemma}\label{P6}
Let $p,r\geq 1$. Let $w_{1},\ldots,w_{r}\in\mathbb{R}^{p}$. Let $\zeta_{1},\ldots,\zeta_{N}$ be independent random variables uniformly distributed on the unit circle. Let $h$ be a product of the random variables $\zeta_{1},\ldots,\zeta_{N}$ and their inverses,
possibly with repetition. Let $i(1),\ldots,i(r)\in\{1,\ldots,N\}$ be distinct. Then
\[\big|\{j\in\{1,\ldots,N\}^{\{1,\ldots,p\}}:\mathbb{E}h\zeta_{i(1)}^{j\cdot w_{1}}\cdots\zeta_{i(r)}^{j\cdot w_{r}}\neq 0\}\big|\leq N^{p-\dim\lspan\{w_{1},\ldots,w_{r}\}}.\]
Equivalently, by Lemma \ref{P1},
\[\sum_{j:\{1,\ldots,p\}\to\{1,\ldots,N\}}|\mathbb{E}h\zeta_{i(1)}^{j\cdot w_{1}}\cdots\zeta_{i(r)}^{j\cdot w_{r}}|\leq N^{p-\dim\lspan\{w_{1},\ldots,w_{r}\}}.\]
\end{lemma}
\begin{lemma}\label{P7}
Let $K$ be a finite set. Let $\pi_{1},\pi_{2}$ be partitions of $K$. Let $(v_{k})_{k\in K}$ be a finite collection of vectors in a vector space $V$ such that whenever $(a_{k})_{k\in K}$ are scalars satisfying
\[\sum_{k\in K}a_{k}v_{k}=0,\]
we have $a_{k}=a_{l}$ for all $k\stackrel{\pi_{2}}{\sim}l$. Then
\[\dim\lspan\left\{\sum_{k\in S}v_{k}:S\in\pi_{1}\right\}\geq|\pi_{1}|-|\pi_{1}\vee\pi_{2}|.\]
\end{lemma}
\begin{proof}
Let $(a_{S})_{S\in\pi_{1}}$ be scalars such that
\[\sum_{S\in\pi_{1}}a_{S}\left(\sum_{k\in S}v_{k}\right)=0.\]
For $k\in K$, let $S(k)$ be the block in $\pi_{1}$ containing $k$. Then
\[0=\sum_{S\in\pi_{1}}a_{S}\left(\sum_{k\in S}v_{k}\right)=\sum_{S\in\pi_{1}}\sum_{k\in S}a_{S}v_{k}=\sum_{S\in\pi_{1}}\sum_{k\in S}a_{S(k)}v_{k}=\sum_{k\in K}a_{S(k)}v_{k}.\]
So by assumption, $a_{S(k)}=a_{S(l)}$ for all $k\stackrel{\pi_{2}}{\sim}l$. Hence, $a_{S(k)}=a_{S(l)}$ for all $k\stackrel{\pi_{1}\vee\pi_{2}}{\sim}l$. Therefore,
\[\dim\left\{(a_{S})_{S\in\pi_{1}}:\sum_{S\in\pi_{1}}a_{S}\left(\sum_{k\in S}v_{k}\right)=0\right\}\leq|\pi_{1}\vee\pi_{2}|.\]
Thus, the result follows.
\end{proof}
\begin{lemma}\label{P8}
Let $K\subset L$ be finite sets. Let $\pi$ be a partition of $L$. Let $\lambda$ be a partition of $K$. Then $\lambda\cup\{\{l\}:l\in L\backslash K\}$ is a partition of $L$ and
\[|(\pi\upharpoonright_{K})\vee\lambda|+|\pi|-|\pi\upharpoonright_{K}|=|\pi\vee(\lambda\cup\{\{l\}:l\in L\backslash K\})|.\]
\end{lemma}
\begin{proof}
Let $K'$ be the union of all blocks in $\pi$ that contain an element in $K$. Then
\[|(\pi\vee(\lambda\cup\{\{l\}:l\in L\backslash K\}))\upharpoonright_{K'}|=|(\pi\vee(\lambda\cup\{\{l\}:l\in L\backslash K\}))\upharpoonright_{K}|=|(\pi\upharpoonright_{K})\vee\lambda|,\]
\[(\pi\vee(\lambda\cup\{\{l\}:l\in L\backslash K\}))\upharpoonright_{L\backslash K'}=\pi\upharpoonright_{L\backslash K'},\]
and
\[|\pi\upharpoonright_{L\backslash K'}|=|\pi|-|\pi\upharpoonright_{K'}|\]
Since $K'$ splits the partition $\pi\vee(\lambda\cup\{\{l\}:l\in L\backslash K\})$, we have
\begin{align*}
|\pi\vee(&\lambda\cup\{\{l\}:l\in L\backslash K\})|\\=&
|(\pi\vee(\lambda\cup\{\{l\}:l\in L\backslash K\}))\upharpoonright_{K'}|+|(\pi\vee(\lambda\cup\{\{l\}:l\in L\backslash K\}))\upharpoonright_{L\backslash K'}|\\=&
|(\pi\upharpoonright_{K})\vee\lambda|+|\pi\upharpoonright_{L\backslash K'}|=|(\pi\upharpoonright_{K})\vee\lambda|+|\pi|-|\pi\upharpoonright_{K'}|.
\end{align*}
\end{proof}
\begin{lemma}\label{P9}
Let $\pi_{1},\pi_{2}$ be a partitions of $L$. If $|\pi_{1}\vee\pi_{2}|>\frac{1}{2}|\pi_{2}|$ then there exists a block $S\in\pi_{2}$ that splits $\pi_{1}$.
\end{lemma}
The proof of Lemma \ref{P9} is analogous to the proof of the fact that a partition of $n$ points with more than $\frac{n}{2}$ blocks must contain a singleton block.

Lemma \ref{P9} can be reformulated as
\begin{lemma}\label{P10}
Let $L$ be a finite set. Let $i:L\to\{1,\ldots,N\}$. Let $\rho$ be a partition of $L$. If $|(\ker i)\vee\rho|>\frac{1}{2}|\rho|$ then there exists a block $S\in\rho$ such that $\{i(l):l\in S\}$ and $\{i(l):l\in L\backslash S\}$ are disjoint.
\end{lemma}
\begin{lemma}\label{P11}
Let $U$ be a finite set in $\mathbb{Z}$. For each $k\in(U-1)\cup(U+1)$, define the vector $v_{k}\in\mathbb{R}^{U}$ as
\[v_{k}=\begin{cases}\begin{array}{cc}e_{k+1},&k\in(U-1)\backslash(U+1)\\e_{k+1}-e_{k-1},&k\in(U-1)\cap(U+1)\\-e_{k-1},&k\in(U+1)\backslash(U-1)\end{array}.\end{cases}\]
If\[\sum_{k\in(U-1)\cup(U+1)}a_{k}v_{k}=0\]
then $a_{l-1}=a_{l+1}$ for all $l\in U$.
\end{lemma}
\begin{proof}
Let $l\in U$. Then
\[\sum_{k\in(U-1)\cup(U+1)}a_{k}\langle v_{k},e_{l}\rangle=0.\]
Since $v_{k}\in\lspan\{e_{k-1},e_{k+1}\}$ for all $k\in(U-1)\cup(U+1)$, if $l\neq k+1$ and $l\neq k-1$ then $\langle v_{k},e_{l}\rangle=0$.
Thus, the only values of $k$ for which $\langle v_{k},e_{l}\rangle$ can possibly be nonvanishing are $l-1$ and $l+1$.
Hence,
\[a_{l+1}\langle v_{l+1},e_{l}\rangle+a_{l-1}\langle v_{l-1},e_{l}\rangle=0.\]
Since $l\in U$, we have:
\begin{enumerate}[label=$\bullet$,leftmargin=20pt]
\item $l+1\in U+1$ and $l-1\in U-1$,
\item if $l+1\in U-1$ then $v_{l+1}=e_{l+2}-e_{l}$ and $\langle v_{l+1},e_{l}\rangle=-1$,
\item if $l+1\notin U-1$ then $v_{l+1}=-e_{l}$ and $\langle v_{l+1},e_{l}\rangle=-1$,
\item if $l-1\in U+1$ then $v_{l-1}=e_{l}-e_{l+1}$ and $\langle v_{l-1},e_{l}\rangle=1$,
\item if $l-1\notin U+1$ then $v_{l-1}=e_{l}$ and $\langle v_{l-1},e_{l}\rangle=1$.
\end{enumerate}
In all of the above cases, $\langle v_{l+1},e_{l}\rangle=-1$ and $\langle v_{l-1},e_{l}\rangle=1$. Therefore, $-a_{l+1}+a_{l-1}=0$. So $a_{l-1}=a_{l+1}$.
\end{proof}

\begin{lemma}\label{P12}
Let $U$ be a finite set in $\mathbb{Z}$. Let $\sim$ be the equivalence relation on $(U-1)\cup(U+1)$ generated by $l-1\sim l+1$ $(l\in U)$.
Then this equivalence relation has at most $|(U+1)\backslash(U-1)|$ equivalence classes.
\end{lemma}
\begin{proof}
It suffices to show that every element $k$ of $(U-1)\cup(U+1)$ is related to an element in $(U+1)\backslash(U-1)$. If $k\in U+1$ then $k\sim k-2\in U-1$. So replacing $k$ by $k-2$, if necessary, we may assume that $k\in U-1$. Let $p$ be smallest natural number for which $k+2p\notin U-1$. By minimality, $k+2q\in U-1$ for all $0\leq q\leq p-1$. So $k+2q+1\in U$ so by assumption, $k+2q\sim k+2q+2$ for all $0\leq q\leq p-1$. Therefore,
\[k\sim k+2\sim k+4\sim\ldots\sim k+2p.\]
Since $k+2(p-1)\in U-1$, $k+2p\in U+1$. Hence, $k+2p\in(U+1)\backslash(U-1)$. 
\end{proof}

\begin{lemma}\label{P13}
Let $n\geq 1$. Let $\epsilon_{1},\ldots,\epsilon_{n}\in\{1,*\}$. Let $\sigma$ be the corresponding maximal alternating interval partition.
For each $I\in\sigma$, let
\[L(I)=\{k\in I:\epsilon_{k}=1\}\cup\{k+1:k\in I\text{ and }\epsilon_{k}=*\}.\]
Then $L(I_{1})\cap L(I_{2})=\emptyset$ for all distinct $I_{1},I_{2}\in\sigma$.
\end{lemma}
Here is a quick example:  if $\eps=(1,*,1,*,*,*,1,1,*)$, then
\[
\sigma=\big\{\{1,2,3,4\},\{5\},\{6,7\},\{8,9\}\}
\]
and 
\[
L(\{1,2,3,4\})=\{1,3,5\},\quad L(\{5\})=\{6\},\quad L(\{6,7\})=\{7\},\quad L(\{8,9\})=\{8,10\}.
\]
\begin{proof}[Proof of Lemma~\ref{P13}]
Let $k\in I_{2}$ with $\epsilon_{k}=*$. If $k+1\notin I_{2}$, then since $I_{2}\in\sigma$, by the definition of $\sigma$, $\epsilon_{k+1}=*$.
On the other hand, if $k+1\in I_{2}$, then since $I_{1}$ and $I_{2}$ are disjoint blocks and are, therefore, disjoint, $k+1\notin I_{1}$.
In both cases, we have that either $k+1\notin I_{1}$ or $\epsilon_{k+1}\neq 1$. Hence,
\[\{k+1:k\in I_{2}\text{ and }\epsilon_{k}=*\}\cap\{k\in I_{1}:\epsilon_{k}=1\}=\emptyset.\]
Interchanging the roles of $I_{1}$ and $I_{2}$, we have
\[\{k+1:k\in I_{1}\text{ and }\epsilon_{k}=*\}\cap\{k\in I_{2}:\epsilon_{k}=1\}=\emptyset.\]
Since $I_{1}$ and $I_{2}$ are disjoint,
\[\{k\in I_{1}:\epsilon_{k}=1\}\cap\{k\in I_{2}:\epsilon_{k}=1\}=\emptyset\]
and
\[\{k+1:k\in I_{1}\text{ and }\epsilon_{k}=*\}\cap\{k+1:k\in I_{2}\text{ and }\epsilon_{k}=*\}=\emptyset.\]
Therefore, $L(I_{1})\cap L(I_{2})=\emptyset$.
\end{proof}

\begin{lemma}\label{P14}
Let $\epsilon_{1},\ldots,\epsilon_{n}\in\{1,*\}$
and let $\sigma$ be the corresponding maximal alternating
Let $L(I)$ for $I\in\sigma$ be as defined in Lemma~\ref{P13}.
Let
\[
U=\{2\leq k\leq n:\epsilon_{k-1}=1\text{ and }\epsilon_{k}=*\}.
\]
Let $\sim$ be the equivalence relation on $(U-1)\cup(U+1)$ generated by $l-1\sim l+1$ for $l\in U$. Then
\begin{enumerate}[label=(\roman*)]
\item\label{it:P14.1} every equivalence class of $\sim$ is of the form $L(I)$ for some $I\in\sigma$
\item\label{it:P14.2} for every $l\in\{2,\ldots,n\}\backslash(U\cup(U-1)\cup(U+1))$, there exists $I\in\sigma$ such that $\{l\}=L(I)$.
\end{enumerate}
\end{lemma}
To illustrate, see the example considered after the statement of Lemma~\ref{P13}.
We have $n=9$, $U=\{2,4,9\}$ and
\[
(U-1)\cup(U+1)=\{1,3,5,8,10\}
\]
with $\sim$ generated by
\[
1\sim3,\quad3\sim5,\quad8\sim10.
\]
Thus, in this example,~\ref{it:P14.1} clearly holds.
Moreover, we have
\[
\{2,\ldots,n\}\setminus(U\cup(U-1)\cup(U+1))=\{6,7\}
\]
and in this example,~\ref{it:P14.2} clearly holds as well.
\begin{proof}[Proof of Lemma~\ref{P14}]
We will first prove~\ref{it:P14.1}.
Let\[\sigma_{0}=\{I\in\sigma:I\cap(U-1)\neq\emptyset\}.\]
We want to show that for every $I\in\sigma_{0}$, $L(I)$ is an equivalence class of $\sim$. After proving this, we show that $\cup_{I\in\sigma_{0}}L(I)=(U-1)\cup(U+1)$. This immediately gives the conclusion of~\ref{it:P14.1},
because $\{L(I)\}_{I\in\sigma_{0}}$ is the partition of $(U-1)\cup(U+1)$ that corresponds to the equivalence relation $\sim$.
\begin{enumerate}[label=\arabic*.,labelwidth=3ex,leftmargin=20pt] 
\item We first show that $L(I)\subset(U-1)\cup(U+1)$ for every $I\in\sigma_{0}$. Since $I\in\sigma_{0}$, there exists $l\in I$ such that $\epsilon_{l}=1$ and $\epsilon_{l+1}=*$. Since $I\in\sigma$, by the definition of $\sigma$, we have $l+1\in I$. (In particular, $l$ and $l+1$ are in $I$.)
\begin{enumerate}[label=\Alph*.,labelwidth=3ex,leftmargin=7pt]
\item  Suppose $k\in I$ and $\eps_k=1$.  We will show $k\in (U-1)\cup(U+1)$.
 Since $I$ is an interval of length at least 2, if $k\in I$ then $k+1\in I$ or $k-1\in I$.
If $k+1\in I$, then by the definition of $\sigma$, $\eps_{k+1}=*$ so $k+1\in U$ and $k\in U-1$.
Suppose $k+1\not\in I$.
Then either $k=n$ or $\eps_{k+1}=1$ and in either case, $k+1\notin U$.
Then $k-1\in I$ and $\eps_{k-1}=*$.
If $k-2\in I$, then $\eps_{k-2}=1$ and $k-1\in U$ and $k\in U+1$.
Otherwise, if $k-2\notin I$, then either $k=2$ or $\eps_{k-2}=*$ and we have $I=\{k-1,k\}$.
But $k\notin U$, so $I\cap(U-1)=\emptyset$, contrary to the hypothesis $I\in\sigma_0$.
Thus, we have shown $\{k\in I:\epsilon_{k}=1\}\subset(U-1)\cup(U+1)$.
\item Suppose $k\in I$ and $\eps_k=*$.  We will show $k+1\in(U-1)\cup(U+1)$.
 Since $I\in\sigma$, by the definition of $\sigma$, we must have $\epsilon_{k-1}=1$ unless $k$ is the smallest element of $I$.
\begin{enumerate}[label=\Roman*.,labelwidth=3ex,leftmargin=7pt]
\item If $\epsilon_{k-1}=1$ then $k\in U$ so $k+1\in U+1$.
\item If $k$ is the smallest element of $I$, then $k\leq l$. Since $\epsilon_{k}=*$ and $\epsilon_{l}=1$, $k\neq l$. So $k+1\leq l$ so $k+2\leq l+1\in I$. Since $I$ is an interval, it follows that $k,k+1,k+2\in I$. So $\epsilon_{k}=*$, $\epsilon_{k+1}=1$ and $\epsilon_{k+2}=*$. So $k+2\in U$ so $k+1\in U-1$.
\end{enumerate}
Thus, we have shown $\{k+1:k\in I\text{ and }\epsilon_{k}=*\}\subset(U-1)\cup(U+1)$.
\end{enumerate}
It follows that $L(I)\subset(U-1)\cup(U+1)$.
\item To show that $L(I)$ is an equivalence class of $\sim$,
we will prove that $L(I)$ is preserved by the equivalence relation $\sim$ and that all elements of $L(I)$ are related.
\begin{enumerate}[label=\Alph*.,labelwidth=3ex,leftmargin=7pt]
\item Suppose $k_{0}\in U$ and $k_{0}-1\in L(I)$.
Since $k_0\in U$, we have
\begin{equation}\label{eq:Uk0k0-1}
\epsilon_{k_{0}-1}=1,\qquad\epsilon_{k_{0}}=*
\end{equation}
Since $k_0-1\in L(I)$, either $k_{0}-1\in\{k\in I:\epsilon_{k}=1\}$ or $k_{0}-1\in\{k+1:k\in I\text{ and }\epsilon_{k}=*\}$.
In the first case, $k_{0}-1\in I$ and $\epsilon_{k_{0}}=*$ so $k_{0}\in I$.
In the second case, $k_{0}-2\in I$ and $\epsilon_{k_{0}-2}=*$; by~\eqref{eq:Uk0k0-1}, we have $k_{0}\in I$.
In both cases, $k_{0}\in I$ and $\eps_{k_0}=*$, so $k_{0}+1\in L(I)$.

On the other hand, if $k_{0}\in U$ and $k_{0}+1\in L(I)$ then either $k_{0}+1\in\{k\in I:\epsilon_{k}=1\}$
or $k_0+1\in\{k+1:k\in I\text{ and }\epsilon_{k}=*\}$.
In the first case, $k_{0}+1\in I$ and $\epsilon_{k_{0}+1}=1$.
Using~\eqref{eq:Uk0k0-1}, $k_{0}-1\in I$.
In the second case, $k_{0}\in I$.
By~\eqref{eq:Uk0k0-1}, $k_{0}-1\in I$.
In both cases $k_{0}-1\in I$ and $\eps_{k_0-1}=1$, so $k_{0}-1\in L(I)$.

Therefore, $L(I)$ is preserved by the equivalence relation $\sim$.
\item To prove that all elements of $L(I)$ are related, note that since $I$ is an interval with alternating values of $\epsilon_{k}$, $\{k\in I:\epsilon_{k}=1\}$ is of the form $\{k_{0},k_{0}+2,\ldots,k_{0}+2p\}$ for some $p\geq 0$ where $\epsilon_{k_{0}}=1$, $\epsilon_{k_{0}+1}=*$, $\epsilon_{k_{0}+2}=1$,$\ldots$, $\epsilon_{k_{0}+2p-1}=*$, $\epsilon_{k_{0}+2p}=1$. Thus, $k_{0}+1,k_{0}+3,\ldots,k_{0}+2p-1\in U$. Thus,
\[k_{0}\sim k_{0}+2\sim k_{0}+4\sim\ldots\sim k_{0}+2p.\]
This means that all the elements in $\{k\in I:\epsilon_{k}=1\}$ are related. Using the same argument, one can show that all the elements in $\{k+1:k\in I\text{ and }\epsilon_{k}=*\}$ are related. Just as the beginning of the first part of the proof, since $I\in\sigma_{0}$, there exists $l\in I$ such that $\epsilon_{l}=1$ and $\epsilon_{l+1}=*$ (thus also $l+1\in I$). So $l\in\{k\in I:\epsilon_{k}=1\}$ and $l+2\in\{k+1:k\in I\text{ and }\epsilon_{k}=*\}$. Since $l+1\in U$, $l\sim l+2$. Therefore, all elements in $L(I)=\{k\in I:\epsilon_{k}=1\}\cup\{k+1:k\in I\text{ and }\epsilon_{k}=*\}$ are related.
\end{enumerate}
Therefore, $L(I)$ is an equivalence class of $\sim$ for every $I\in\sigma_{0}$.
\item It remains to show that $\cup_{I\in\sigma_{0}}L(I)=(U-1)\cup(U+1)$. Since $L(I)\subset(U-1)\cup(U+1)$ by the first part of the proof, it suffices to show that $(U-1)\cup(U+1)\subset\cup_{I\in\sigma_{0}}L(I)$.
\begin{enumerate}[label=\Roman*.,labelwidth=3ex,leftmargin=7pt]
\item If $k_{0}\in U-1$ then $\epsilon_{k_{0}}=1$. Let $I\in\sigma$ contain $k_{0}$. Then $k_{0}\in L(I)$ and $I\in\sigma_{0}$.
\item If $k_{0}\in U+1$ then $\epsilon_{k_{0}-2}=1$ and $\epsilon_{k_{0}-1}=*$.
Let $I\in\sigma$ contain $k_{0}-1$.
Then $k_{0}\in L(I)$ and $I\in\sigma_{0}$, since $k_{0}-2\in I\cap(U-1)$.
\end{enumerate}
\end{enumerate}
This completes the proof of~\ref{it:P14.1}.

We now prove~\ref{it:P14.2}. Let $l\in\{2,\ldots,n\}\backslash(U\cup(U-1)\cup(U+1))$.
\begin{enumerate}[label=\arabic*.,labelwidth=3ex,leftmargin=20pt]
\item If $\epsilon_{l}=1$, then since $l\notin U-1$, either $l=n$ or $\epsilon_{l+1}=1$.
Since $l\notin U+1$, either $l=2$ or $\epsilon_{l-2}=*$ or $\epsilon_{l-1}=1$.
\begin{enumerate}[label=\Alph*.,labelwidth=3ex,leftmargin=7pt]
\item If $l<n$ and $\epsilon_{l-1}=1$,
then $\epsilon_{l-1}=\epsilon_{l}=\epsilon_{l+1}=1$, which implies $\{l\}\in\sigma$.
Moreover, since $\epsilon_{l}=1$, $L(\{l\})=\{l\}$.
\item If $l=n$ and $\epsilon_{l-1}=1$,
then similarly, $\epsilon_{n-1}=\epsilon_{n}=1$ and we have $\{n\}\in\sigma$ and $L(\{n\})=\{n\}$.
\item If $\epsilon_{l-1}=*$ and $2<l<n$, then $\epsilon_{l-2}=*$ and, since $\epsilon_{l}=\epsilon_{l+1}=1$, we have
$\{l-1,l\}\in\sigma$ and $L(\{l-1,l\})=\{l\}$.
\item If $e_{n-1}=*$ and if $2<l=n$ or $2=l<n$, then similarly and $\{l-1,l\}\in\sigma$ and $L(\{l-1,l\})=\{l\}$.
\item If $2=l=n$ and $e_1=*$, then $\{1,2\}\in\sigma$ and $L(\{1,2\})=\{2\}$.
\end{enumerate}
\item If $\epsilon_{l}=*$, then since $l\notin U$ we have $\epsilon_{l-1}=*$.
Since $l\notin U+1$, either $l=2$ or $\epsilon_{l-2}=*$.
In either case, we have $\{l-1\}\in\sigma$ and $L(\{l-1\})=\{l\}$.
\end{enumerate}
 This completes the proof.
\end{proof}
In the sequel, if $A$ is a $N\times N$ random matrix and $p\geq 1$ then
\[|A|_{p}:=(\mathbb{E}\circ\mathrm{tr}(A^{*}A)^{\frac{p}{2}})^{\frac{1}{p}}.\]
Thus, if $A$ is deterministic then $|A|_{p}=(\mathrm{tr}(A^{*}A)^{\frac{p}{2}})^{\frac{1}{p}}$ is the normalized Schatten $p$ norm.
\begin{lemma}\label{P15}
Let $A$ be a $N\times N$ random matrix with integrable entries. Let $p\geq 1$. Then
\[|\mathbb{E}A|_{p}\leq|A|_{p}.\]
\end{lemma}
\begin{proof}
Since $|\cdot|_{p}$ is a norm on deterministic $N\times N$ matrices,
\[|\mathbb{E}A|_{p}\leq\mathbb{E}(\mathrm{tr}(A^{*}A)^{\frac{p}{2}})^{\frac{1}{p}}\leq(\mathbb{E}\circ\mathrm{tr}(A^{*}A)^{\frac{p}{2}})^{\frac{1}{p}}=|A|_{p},\]
where the first inequality follows from Jensen's inequality and the second inequality follows from H\"older's inequality.
\end{proof}
\begin{lemma}[\cite{Bhatia}, Exercise IV.2.7]\label{P16}
Let $A_{1}$ and $A_{2}$ be $N\times N$ (deterministic) matrices. Let $p,q,r$ be positive real numbers such that $\frac{1}{p}+\frac{1}{q}=\frac{1}{r}$. Then
\[|A_{1}A_{2}|_{r}\leq|A_{1}|_{p}|A_{2}|_{q}.\]
\end{lemma}
Applying Lemma~\ref{P16} repeatedly, one obtains
\begin{lemma}\label{P17}
Let $A_{1},\ldots,A_{s}$ be $N\times N$ (deterministic) matrices. Let $p_{1},\ldots,p_{s},r\geq 1$ be such that $\frac{1}{p_{1}}+\ldots+\frac{1}{p_{s}}=\frac{1}{r}$. Then
\[|A_{1}\ldots A_{s}|_{r}\leq|A_{1}|_{p_{1}}\ldots|A_{s}|_{p_{s}}.\]
\end{lemma}
Applying the above to random matrices, we get the following:
\begin{lemma}\label{P18}
Let $A_{1},\ldots,A_{s}$ be $N\times N$ random matrices 
having finite moments of all orders.
Let $p_{1},\ldots,p_{s},r\geq 1$ be such that $\frac{1}{p_{1}}+\cdots+\frac{1}{p_{s}}=\frac{1}{r}$. Then
\begin{equation}\label{eq:EHoelder}
|A_{1}\ldots A_{s}|_{r}\leq|A_{1}|_{p_{1}}\ldots|A_{s}|_{p_{s}}.
\end{equation}
\end{lemma}
\begin{proof}
By Lemma \ref{P17},
\[\mathrm{tr}((A_{1}\cdots A_{s})^{*}(A_{1}\cdots A_{s}))^{\frac{r}{2}}\leq(\mathrm{tr}(A_{1}^{*}A_{1})^{\frac{p_{1}}{2}})^{\frac{r}{p_{1}}}\cdots(\mathrm{tr}(A_{s}^{*}A_{s})^{\frac{p_{s}}{2}})^{\frac{r}{p_{s}}}.\]
Taking expectations and using H\"older's inequality, we obtain
\begin{eqnarray*}
\mathbb{E}\circ\mathrm{tr}((A_{1}\cdots A_{s})^{*}(A_{1}\cdots A_{s}))^{\frac{r}{2}}&\leq&\mathbb{E}((\mathrm{tr}(A_{1}^{*}A_{1})^{\frac{p_{1}}{2}})^{\frac{r}{p_{1}}}\cdots(\mathrm{tr}(A_{s}^{*}A_{s})^{\frac{p_{s}}{2}})^{\frac{r}{p_{s}}})\\&\leq&(\mathbb{E}\circ\mathrm{tr}(A_{1}^{*}A_{1})^{\frac{p_{1}}{2}})^{\frac{r}{p_{1}}}\cdots(\mathbb{E}\circ\mathrm{tr}(A_{s}^{*}A_{s})^{\frac{p_{s}}{2}})^{\frac{r}{p_{s}}}\\&=&|A_{1}|_{p_{1}}^{r}\cdots|A_{s}|_{p_{s}}^{r}.
\end{eqnarray*}
Thus,~\eqref{eq:EHoelder} holds.
\end{proof}
\begin{lemma}\label{P19}
Let $A_{1}^{(1)},\ldots,A_{s}^{(1)},A_{1}^{(2)},\ldots,A_{s}^{(2)}$ be $N\times N$ random matrices
having finite moments of all orders.
Let $M=\max\{|A_{l}^{(1)}|_{2(s-1)},|A_{l}^{(2)}|_{2(s-1)}:1\leq l\leq s\}$ if $s\geq 2$ and let $M=1$ if $s=1$. Then
\[\left|\mathbb{E}\circ\mathrm{tr}\left(\prod_{l=1}^{s}\left(A_{l}^{(1)}+A_{l}^{(2)}\right)\right)-\mathbb{E}\circ\mathrm{tr}\left(\prod_{l=1}^{s}A_{l}^{(1)}\right)\right|\leq 2^{s}M^{s-1}\max_{1\leq l\leq s}|A_{l}^{(2)}|_{2}.\]
\end{lemma}
\begin{proof}
If $s=1$, the result follows from the Cauchy--Schwarz inequality.
Assume $s\geq 2$. First,
\[\mathbb{E}\circ\mathrm{tr}\left(\prod_{l=1}^{s}\left(A_{l}^{(1)}+A_{l}^{(2)}\right)\right)=\sum_{\epsilon_{1},\ldots,\epsilon_{s}\in\{1,2\}}\mathbb{E}\circ\mathrm{tr}(A_{1}^{(\epsilon_{1})}\cdots A_{s}^{\epsilon_{s}}).\]
So
\begin{equation}\label{Product}
\left|\mathbb{E}\circ\mathrm{tr}\left(\prod_{l=1}^{s}\left(A_{l}^{(1)}+A_{l}^{(2)}\right)\right)-\mathbb{E}\circ\mathrm{tr}\left(\prod_{l=1}^{s}A_{l}^{(1)}\right)\right|\leq\sum_{\substack{\epsilon_{1},\ldots,\epsilon_{s}\in\{1,2\}\\\exists l_{0}\text{ s.t. }\epsilon_{l_{0}}\text{ is }2}}|\mathbb{E}\circ\mathrm{tr}(A_{1}^{(\epsilon_{1})}\cdots A_{s}^{(\epsilon_{s})})|.
\end{equation}
For each $\epsilon_{1},\ldots,\epsilon_{s}\in\{1,2\}$ with $\epsilon_{l_{0}}=2$, taking $p_{l}=2(s-1)$ for $l\neq l_{0}$, $p_{l_{0}}=2$ and $r=1$ in Lemma \ref{P17}, we obtain
\begin{align*}
\big|\mathbb{E}\circ\mathrm{tr}(A_{1}^{(\epsilon_{1})}\cdots A_{s}^{(\epsilon_{s})})\big|
&\leq\big|A_{l_{0}}^{(2)}\big|_{2}\,\big|A_{l_0+1}^{(\eps_{l_0+1})}\cdots A_s^{(\eps_s)}A_1^{(\eps_1)}\cdots A_{l_0-1}^{(\eps_{l_0-1})}\big|_2 \\[1ex]
&\leq\big|A_{l_{0}}^{(2)}\big|_{2}\,\prod_{j\ne l_0}\big|A_j^{(\eps_j)}\big|_{2(s-1)}\le M^{s-1}\big|A_{l_{0}}^{(2)}\big|_{2},
\end{align*}
where for the first inequality we used the trace property and the Cauchy--Schwarz inequality, while for the second we used
H\"older's inequality (Lemma~\ref{P18}).
Since there are $2^s-1$ terms in the summation in~\eqref{Product}, the desired upper bound holds.
\end{proof}

\medskip
We will now show that the off-diagonal entries of alternating products in $X_N$ and $X_N^*$,
with deterministic diagonal matrices interspersed,
have expectations that are zero or are asymptotically small as the matrix size goes to infinity.

\begin{lemma}
Let $\zeta_{1},\ldots,\zeta_{N}$ be independent random variables uniformly distributed on the unit circle. Let $i(1),\ldots,i(r)\in\{1,\ldots,N\}$. If $j(1),\ldots,j(r)\in\mathbb{Z}$ satisfy
\[\mathbb{E}\zeta_{i(1)}^{j(1)}\cdots\zeta_{i(r)}^{j(r)}\neq 0\]
then $j(1)+\cdots+j(r)=0$.
\end{lemma}
\begin{proof}
Let $\pi=\ker i$. By Lemma~\ref{P1}, $\sum_{k\in S}j(k)=0$ for all $S\in\pi$.
So
\[
\sum_{k=1}^{r}j(k)=\sum_{S\in\pi}\sum_{k\in S}j(k)=0.
\]
\end{proof}
\begin{lemma}\label{Odd}
Let $n\geq 1$ be an odd number. Let $\epsilon_{1},\ldots,\epsilon_{n}\in\{1,*\}$ be alternating. Let $d_{1},\ldots,d_{n}$ be deterministic diagonal $N\times N$ matrices. Then
\[\mathbb{E}\prod_{k=1}^{n}d_{k}X_{N}^{\epsilon_{k}}=0.\]
\end{lemma}
\begin{proof}
The proof when $\epsilon_{1}=1$ and the proof when $\epsilon_{1}=*$ are similar. So we only do the case when $\epsilon_{1}=1$. Let $i(1),i(n+1)\in\{1,\ldots,N\}$.
\begin{eqnarray}\label{Exp}
\left(\mathbb{E}\prod_{k=1}^{n}d_{k}X_{N}^{\epsilon_{k}}\right)_{i(1),i(n+1)}&=&
\sum_{i:\{2,3,\ldots,n\}\to\{1,\ldots,N\}}\mathbb{E}\prod_{k=1}^{n}(d_{k})_{i(k),i(k)}(X_{N}^{\epsilon_{k}})_{i(k),i(k+1)}\nonumber\\&=&
\sum_{i:\{2,3,\ldots,n\}\to\{1,\ldots,N\}}\prod_{k=1}^{n}(d_{k})_{i(k),i(k)}\mathbb{E}\prod_{k=1}^{n}(X_{N}^{\epsilon_{k}})_{i(k),i(k+1)}.
\end{eqnarray}
Since $\epsilon_{1}=1$ and $\epsilon_{1},\ldots,\epsilon_{n}$ are alternating, $\epsilon_{k}=1$ when $k$ is odd, and $\epsilon_{k}=*$ when $k$ is even so
\begin{eqnarray*}
\mathbb{E}\prod_{k=1}^{n}(X_{N}^{\epsilon_{k}})_{i(k),i(k+1)}&=&
\frac{1}{N^{\frac{n}{2}}}\mathbb{E}\prod_{l=1}^{(n+1)/2}(X_{N}^{\epsilon_{2l-1}})_{i(2l-1),i(2l)}\prod_{m=1}^{(n-1)/2}(X_{N}^{\epsilon_{2m}})_{i(2m),i(2m+1)}\\&=&
\frac{1}{N^{\frac{n}{2}}}\mathbb{E}\prod_{l=1}^{(n+1)/2}\zeta_{i(2l-1)}^{i(2l)}\prod_{m=1}^{(n-1)/2}\zeta_{i(2m+1)}^{-i(2m)}.
\end{eqnarray*}
Since the sum of the exponents is
\[\sum_{l=1}^{(n+1)/2}i(2l)+\sum_{m=1}^{(n-1)/2}(-i(2m))=i(n+1)\neq 0,\]
by Lemma \ref{Odd}, $\mathbb{E}\prod_{k=1}^{n}(X_{N}^{\epsilon_{k}})_{i(k),i(k+1)}=0$. Thus, the result follows.
\end{proof}
Using an argument similar to that in the proof of Lemma \ref{Odd}, one obtains
\begin{lemma}\label{*1}
Let $n\geq2$ be an even number. Let $\epsilon_{1},\ldots,\epsilon_{n}\in\{1,*\}$ be alternating. Let $d_{1},\ldots,d_{n}$ be deterministic diagonal $N\times N$ matrices. If $\epsilon_{1}=*$ then
\[\mathbb{E}\prod_{k=1}^{n}d_{k}X_{N}^{\epsilon_{k}}\]
is a diagonal matrix.
\end{lemma}
\begin{lemma}\label{1*}
Let $n\geq 1$ be an even number. Let $\epsilon_{1},\ldots,\epsilon_{n}\in\{1,*\}$ be alternating. Suppose that $\epsilon_{1}=1$. Let $d_{1},\ldots,d_{n}$ be deterministic diagonal $N\times N$ matrices of norm at most 1. Let
\[Z_{N}=\prod_{k=1}^{n}d_{k}X_{N}^{\epsilon_{k}}.\]
Then for every integer $p\geq 1$, there is a constant $C=C(n,p)$ such that
\begin{equation}\label{eq:TrCpn}
\mathrm{Tr}((\mathbb{E}Z_{N}-\mathbb{E}\circ\mathrm{diag}Z_{N})^{*}(\mathbb{E}Z_{N}-\mathbb{E}\circ\mathrm{diag}Z_{N}))^{p}\leq C.
\end{equation}
\end{lemma}
\begin{proof}
Let $i(1)\neq i(n+1)\in\{1,\ldots,N\}$. By~\eqref{Exp},
\begin{eqnarray*}
(\mathbb{E}Z_{N})_{i(1),i(n+1)}&=&\left(\mathbb{E}\prod_{k=1}^{n}d_{k}X_{N}^{\epsilon_{k}}\right)_{i(1),i(n+1)}\\&=&\sum_{i:\{2,3,\ldots,n\}\to\{1,\ldots,N\}}\prod_{k=1}^{n}(d_{k})_{i(k),i(k)}\mathbb{E}\prod_{k=1}^{n}(X_{N}^{\epsilon_{k}})_{i(k),i(k+1)}.
\end{eqnarray*}
Since the $d_{k}$ have norms at most 1, we have
\begin{equation}\label{Expandineq}
|(\mathbb{E}Z_{N})_{i(1),i(n+1)}|\leq\sum_{i:\{2,3,\ldots,n\}\to\{1,\ldots,N\}}\left|\mathbb{E}\prod_{k=1}^{n}(X_{N}^{\epsilon_{k}})_{i(k),i(k+1)}\right|.
\end{equation}
Since $\epsilon_{k}=1$ when $k$ is odd and $\epsilon_{k}=*$ when $k$ is even,
\begin{align*}
\mathbb{E}\prod_{k=1}^{n}&(X_{N}^{\epsilon_{k}})_{i(k),i(k+1)}=
\mathbb{E}\prod_{l=1}^{n/2}(X_{N}^{\epsilon_{i(2l-1)}})_{i(2l-1),i(2l)}\prod_{m=1}^{n/2}(X_{N}^{\epsilon_{2m}})_{i(2m),i(2m+1)} \\
&=\frac{1}{N^{\frac{n}{2}}}\mathbb{E}\prod_{l=1}^{n/2}\zeta_{i(2l-1)}^{i(2l)}\prod_{m=1}^{n/2}\zeta_{i(2m+1)}^{-i(2m)}
=\frac{1}{N^{\frac{n}{2}}}\mathbb{E}\prod_{l=1}^{n/2}\zeta_{i(2l-1)}^{i(2l)}\prod_{m=2}^{(n/2)+1}\zeta_{i(2m-1)}^{-i(2m-2)}\\
&=\frac{1}{N^{\frac{n}{2}}}\mathbb{E}\zeta_{i(1)}^{i(2)}\left(\prod_{l=2}^{n/2}\zeta_{i(2l-1)}^{i(2l)}\right)\left(\prod_{m=2}^{n/2}\zeta_{i(2m-1)}^{-i(2m-2)}\right)\zeta_{i(n+1)}^{-i(n)} \\
&=\frac{1}{N^{\frac{n}{2}}}\mathbb{E}\zeta_{i(1)}^{i(2)}\left(\prod_{l=2}^{n/2}\zeta_{i(2l-1)}^{i(2l)-i(2l-2)}\right)\zeta_{i(n+1)}^{-i(n)}.
\end{align*}
Let $v_{1},v_{3},\ldots,v_{n+1}\in\mathbb{R}^{n}$ be given by
\begin{align*}
v_{1}&=e_{2} \\
v_{2l-1}&=e_{2l}-e_{2l-2}, \quad(l=2,\ldots,\frac{n}{2}) \\
v_{n+1}&=-e_{n}.
\end{align*}
Let $j:\{2,4,\ldots,n\}\to\{1,\ldots,N\}$ be the restriction of $i$ to $\{2,4\ldots,n\}$.
Then we have
\[\mathbb{E}\prod_{k=1}^{n}(X_{N}^{\epsilon_{k}})_{i(k),i(k+1)}=\frac{1}{N^{\frac{n}{2}}}\mathbb{E}\prod_{l=1}^{(n/2)+1}\zeta_{i(2l-1)}^{j\cdot v_{2l-1}}=\frac{1}{N^{\frac{n}{2}}}\mathbb{E}\prod_{k\in\{1,3,\ldots,n+1\}}\zeta_{i(k)}^{j\cdot v_{k}}.\]
Let $\pi$ be a partition of $\{1,3,\ldots,n+1\}$. Suppose that $\ker(i\upharpoonright_{\{1,3,\ldots,n+1\}})=\pi$. For each $S\in\pi$, all the $i(k)$ are same for $k\in S$ and we denote this value by $i(S)$. Thus,
\[\mathbb{E}\prod_{k=1}^{n}(X_{N}^{\epsilon_{k}})_{i(k),i(k+1)}=\frac{1}{N^{\frac{n}{2}}}\mathbb{E}\prod_{S\in\pi}\zeta_{i(S)}^{j\cdot(\sum_{k\in S}v_{k})}.\]
By Lemma \ref{P6},
\[\sum_{j:\{2,4,\ldots,n\}\to\{1,\ldots,N\}}\left|\mathbb{E}\prod_{k=1}^{n}(X_{N}^{\epsilon_{k}})_{i(k),i(k+1)}\right|\leq\frac{1}{N^{\frac{n}{2}}}N^{\frac{n}{2}-\dim\lspan\{\sum_{k\in S}v_{k}:S\in\pi\}}.\]
If $a_{1},a_{3},\ldots,a_{n+1}$ are scalars satisfying
\[a_{1}v_{1}+a_{3}v_{3}+\ldots+a_{n+1}v_{n+1}=0,\]
then $a_{1}=a_{3}=\ldots=a_{n+1}$. Thus, by Lemma \ref{P7}, $\dim\lspan\{\sum_{k\in S}v_{k}:S\in\pi\}\geq|\pi|-1$ so
\begin{eqnarray*}
\sum_{j:\{2,4,\ldots,n\}\to\{1,\ldots,N\}}\left|\mathbb{E}\prod_{k=1}^{n}(X_{N}^{\epsilon_{k}})_{i(k),i(k+1)}\right|&\leq&
\frac{1}{N^{\frac{n}{2}}}N^{\frac{n}{2}-(|\pi|-1)}\\&=&N^{1-|\pi|}.
\end{eqnarray*}
Considering all the cases when  $\{1\}$ is or is not a singleton block and $\{n+1\}$ is or is not a singleton block of $\pi$,
we see that the number of choices of $i(3),i(5),\ldots,i(n-1)$ such that $\ker(i\upharpoonright_{\{1,3,\ldots,n+1\}})=\pi$
is at most $N^{|\pi|-2}$ and, thus,
\[\sum_{\substack{i:\{2,\ldots,n\}\to\{1,\ldots,N\}\\\ker(i\upharpoonright_{\{1,3,\ldots,n+1\}})=\pi}}\left|\mathbb{E}\prod_{k=1}^{n}(X_{N}^{\epsilon_{k}})_{i(k),i(k+1)}\right|\leq N^{|\pi|-2}N^{1-|\pi|}=\frac{1}{N}.\]
Summing over all partitions $\pi$ of $\{1,3,\ldots,n+1\}$, we have
\[\sum_{i:\{2,\ldots,n\}\to\{1,\ldots,N\}}\left|\mathbb{E}\prod_{k=1}^{n}(X_{N}^{\epsilon_{k}})_{i(k),i(k+1)}\right|\leq\frac{C_n}{N}.\]
So by \eqref{Expandineq},
\[|(\mathbb{E}Z_{N})_{i(1),i(n+1)}|\leq\frac{C_n}{N}.\]
So each entry in $\mathbb{E}(Z_{N}-\mathrm{diag}Z_{N})$ has absolute value at most $C_n/N$.
From this, the result follows easily.
Indeed, each entry of 
\[
(\mathbb{E}Z_{N}-\mathbb{E}\circ\mathrm{diag}Z_{N})^{*}(\mathbb{E}Z_{N}-\mathbb{E}\circ\mathrm{diag}Z_{N})
\]
has absolute value at most $C_n^2/N$.
Taken to the $p$-th power, every entry has absolute value at most $C_n^{2p}/N$, and the result~\eqref{eq:TrCpn} follows
with constant $C=C_n^{2p}$.
\end{proof}
Using a similar argument as in the proof of Lemma \ref{1*}, (essentially, by treating also the case $i(n+1)=i(1)$ in that proof)
one obtains the following lemma.
\begin{lemma}[Compare with Proposition 1 in \cite{RyanDebbah}]\label{Bound}
For every integer $p\geq 1$, we have $\mathbb{E}\circ\mathrm{Tr}(X_{N}^{*}X_{N})^{p}\leq CN$, where $C$ depends only on $p$.
\end{lemma}

\begin{lemma}\label{Offdiag}
Let $\epsilon_{1},\ldots,\epsilon_{n}\in\{1,*\}$ be alternating. Let $d_{1},\ldots,d_{n}$ be deterministic diagonal $N\times N$ matrices of norm at most 1. Let
\[Z_{N}=\prod_{k=1}^{n}d_{k}X_{N}^{\epsilon_{k}}\]
Then for every integer $p\geq 1$,
\[|Z_{N}|_{2p}\leq C\quad\text{ and }\quad|\mathbb{E}Z_{N}-\mathbb{E}\circ\mathrm{diag}Z_{N}|_{2p}\leq\frac{C}{N^{\frac{1}{2p}}},\]
where $C$ depends only on $n$ and $p$.
\end{lemma}
\begin{proof}
By Lemma \ref{Bound}, for every integer $q\geq 1$,
\[|X_{N}|_{2q}\leq C_q,\]
where $C_q$ depends only on $q$.
Thus, taking $p_{1}=\cdots=p_{n}=2pn$, $r=2p$ in Lemma \ref{P18}, we have
\begin{eqnarray*}
|Z_{N}|_{2p}=\left|\prod_{k=1}^{n}d_{k}X_{N}^{\epsilon_{k}}\right|_{2p}&\leq&\prod_{k=1}^{n}|d_{k}X_{N}^{\epsilon_{k}}|_{2pn}\\&\leq&
\prod_{k=1}^{n}\|d_{k}\||X_{N}^{\epsilon_{k}}|_{2pn}\leq\prod_{k=1}^{n}|X_{N}^{\epsilon_{k}}|_{2pn}\leq\prod_{k=1}^{n}C_{2pn}.
\end{eqnarray*}
This proves the first inequality. The other inequality follows by combining Lemmas \ref{Odd}, \ref{*1} and \ref{1*}.
\end{proof}

\medskip
We are now ready to prove Proposition~\ref{OneProp}.
We first prove a weaker version of it, with $\mathbb{E}\circ\mathrm{diag}$ replaced by $\mathbb{E}$.
The convention regarding ordering in products is described in Definition~\ref{def:rdiag}.
\begin{lemma}\label{OneLem}
Let $n\geq 1$. Let $\epsilon_{1},\ldots,\epsilon_{n}\in\{1,*\}$. Let $\sigma$ be the interval partition of $\{1,\ldots,n\}$ defined by
\[k\stackrel{\sigma}{\sim}k+1\Longleftrightarrow\epsilon_{k}\neq\epsilon_{k+1}\]
for $k\in\{1,\ldots,n-1\}$. Let $d_{1},\ldots,d_{n}$ be deterministic diagonal $N\times N$ matrices of norm at most 1. Then
\[\left|\mathbb{E}\circ\mathrm{Tr}\prod_{S\in\sigma}\left(\prod_{k\in S}d_{k}X_{N}^{\epsilon_{k}}-\mathbb{E}\left(\prod_{k\in S}d_{k}X_{N}^{\epsilon_{k}}\right)\right)\right|\leq C\sqrt{N},\]
where $C$ depends only on $n$.
\end{lemma}
\begin{proof}
Since $\sigma$ is an interval partition, we can expand
\begin{align*}
\mathbb{E}\circ\mathrm{Tr}&\prod_{S\in\sigma}\left(\prod_{k\in S}d_{k}X_{N}^{\epsilon_{k}}-\mathbb{E}\left(\prod_{k\in S}d_{k}X_{N}^{\epsilon_{k}}\right)\right)\\
=&\sum_{\substack{i:\{1,\ldots,n+1\}\to\{1,\ldots,N\}\\i(n+1)=i(1)}}
\begin{aligned}[t]
\Bigg(\mathbb{E}&\prod_{S\in\sigma}\Bigg(\prod_{k\in S}(d_{k})_{i(k),i(k)}(X_{N}^{\epsilon_{k}})_{i(k),i(k+1)} \\
&-\mathbb{E}\left(\prod_{k\in S}(d_{k})_{i(k),i(k)}(X_{N}^{\epsilon_{k}})_{i(k),i(k+1)}\right)\Bigg)\Bigg)
\end{aligned}
\\=&\sum_{\substack{i:\{1,\ldots,n+1\}\to\{1,\ldots,N\}\\i(n+1)=i(1)}}\prod_{k=1}^{n}(d_{k})_{i(k),i(k)}
\begin{aligned}[t]
\Bigg(\mathbb{E}&\prod_{S\in\sigma}\Bigg(\prod_{k\in S}(X_{N}^{\epsilon_{k}})_{i(k),i(k+1)} \\
&-\mathbb{E}\left(\prod_{k\in S}(X_{N}^{\epsilon_{k}})_{i(k),i(k+1)}\right)\Bigg).
\end{aligned}
\end{align*}
Since the $d_{k}$ have norms at most 1, it follows that
\begin{multline}\label{Onealign}
\left|\mathbb{E}\circ\mathrm{Tr}\prod_{S\in\sigma}\left(\prod_{k\in S}d_{k}X_{N}^{\epsilon_{k}}
-\mathbb{E}\left(\prod_{k\in S}d_{k}X_{N}^{\epsilon_{k}}\right)\right)\right| \\
\leq
\sum_{i:\{1,\ldots,n+1\}\to\{1,\ldots,N\}}\delta_{i(n+1),i(1)}
\begin{aligned}[t]
\Bigg|\mathbb{E}&\prod_{S\in\sigma}\Bigg(\prod_{k\in S}(X_{N}^{\epsilon_{k}})_{i(k),i(k+1)} \\
&-\mathbb{E}\left(\prod_{k\in S}(X_{N}^{\epsilon_{k}})_{i(k),i(k+1)}\right)\Bigg)\Bigg|.
\end{aligned}
\end{multline}
For each $i:\{1,\ldots,n+1\}\to\{1,\ldots,N\}$, by Lemma \ref{P2}, we have
\begin{multline}\label{Estimate}
\left|\mathbb{E}\prod_{S\in\sigma}\left(\prod_{k\in S}(X_{N}^{\epsilon_{k}})_{i(k),i(k+1)}-\mathbb{E}\left(\prod_{k\in S}(X_{N}^{\epsilon_{k}})_{i(k),i(k+1)}\right)\right)\right| \\
\leq\mathbb{E}\prod_{S\in\sigma}\left(\prod_{k\in S}(X_{N}^{\epsilon_{k}})_{i(k),i(k+1)}\right)
=\mathbb{E}\prod_{k=1}^{n}(X_{N}^{\epsilon_{k}})_{i(k),i(k+1)}.
\end{multline}
Let\[U=\{2\leq k\leq n:\epsilon_{k-1}=1\text{ and }\epsilon_{k}=*\}.\]
Let $Y_{N}=\sqrt{N}X_{N}$. Then
\begin{align*}
\mathbb{E}\prod_{k=1}^{n}&(Y_{N}^{\epsilon_{k}})_{i(k),i(k+1)}\\
&=\mathbb{E}\prod_{k\in U}(Y_{N}^{\epsilon_{k}})_{i(k),i(k+1)}\prod_{k\in U-1}(Y_{N}^{\epsilon_{k}})_{i(k),i(k+1)}
\prod_{k\in\{1,\ldots,n\}\backslash(U\cup(U-1))}(Y_{N}^{\epsilon_{k}})_{i(k),i(k+1)} \displaybreak[2]\\
&=\mathbb{E}\prod_{k\in U}(Y_{N}^{\epsilon_{k}})_{i(k),i(k+1)}\prod_{k\in U}(Y_{N}^{\epsilon_{k-1}})_{i(k-1),i(k)}
\prod_{k\in\{1,\ldots,n\}\backslash(U\cup(U-1))}(Y_{N}^{\epsilon_{k}})_{i(k),i(k+1)}\displaybreak[2]\\
&=\mathbb{E}\prod_{k\in U}\zeta_{i(k+1)}^{-i(k)}\prod_{k\in U}\zeta_{i(k-1)}^{i(k)}
\prod_{k\in\{1,\ldots,n\}\backslash(U\cup(U-1))}(Y_{N}^{\epsilon_{k}})_{i(k),i(k+1)}\displaybreak[2]\\
&=\mathbb{E}\prod_{k\in U+1}\zeta_{i(k)}^{-i(k-1)}\prod_{k\in U-1}\zeta_{i(k)}^{i(k+1)}
\prod_{k\in\{1,\ldots,n\}\backslash(U\cup(U-1))}(Y_{N}^{\epsilon_{k}})_{i(k),i(k+1)} \displaybreak[2]\\
&=\mathbb{E}\Bigg(
\begin{aligned}[t]
&\prod_{k\in (U+1)\backslash(U-1)}\zeta_{i(k)}^{-i(k-1)}\prod_{(U-1)\backslash(U+1)}\zeta_{i(k)}^{i(k+1)}\\
&\hspace{2em}\times\prod_{k\in (U+1)\cap(U-1)}\zeta_{i(k)}^{i(k+1)-i(k-1)}\prod_{k\in\{1,\ldots,n\}\backslash(U\cup(U-1))}(Y_{N}^{\epsilon_{k}})_{i(k),i(k+1)}\Bigg).
\end{aligned}
\end{align*}
Let $L=\{1,\ldots,n+1\}\backslash U$. Note that by the definition of $U$, $(U-1)\cup(U+1)\subset L$. Also, if $k\in\{1,\ldots,n+1\}\backslash(U\cup(U-1))$ then $k\in L$ and $k+1\in L$.

Let $i_U:U\to\{1,\ldots,N\}$ be the restriction of $i$ to $U$. Let $i_{L}:L\to\{1,\ldots,N\}$ be the restriction of $i$ to $L$.
With $v_{k}\in\Reals^U$ defined as in Lemma \ref{P11}, we have
\[\mathbb{E}\prod_{k=1}^{n}(Y_{N}^{\epsilon_{k}})_{i(k),i(k+1)}=\mathbb{E}\prod_{k\in(U-1)\cup(U+1)}\zeta_{i_{L}(k)}^{i_U\cdot v_{k}}\prod_{k\in\{1,\ldots,n\}\backslash(U\cup(U-1))}(Y_{N}^{\epsilon_{k}})_{i_{L}(k),i_{L}(k+1)},\]
where we think of $i_U$ as belonging to $\Reals^U$.
Let $\pi$ be a partition of $L$. Suppose that $\ker i_{L}=\pi$. Let $\pi_{1}=\pi\upharpoonright_{(U-1)\cup(U+1)}$. For each block $S\in\pi_{1}$, all the $i_{L}(k)$ are the same for $k\in S$ and we denote this value by $i_{L}(S)$. It follows that
\begin{eqnarray*}
\mathbb{E}\prod_{k=1}^{n}(Y_{N}^{\epsilon_{k}})_{i(k),i(k+1)}&=&\mathbb{E}\prod_{S\in\pi_{1}}\prod_{k\in S}\zeta_{i_{L}(k)}^{i_U\cdot v_{k}}\prod_{k\in\{1,\ldots,n\}\backslash(U\cup(U-1))}(Y_{N}^{\epsilon_{k}})_{i_{L}(k),i_{L}(k+1)}\\&=&
\mathbb{E}\prod_{S\in\pi_{1}}\zeta_{i_{L}(S)}^{i_U\cdot(\sum_{k\in S}v_{k})}\prod_{k\in\{1,\ldots,n\}\backslash(U\cup(U-1))}(Y_{N}^{\epsilon_{k}})_{i_{L}(k),i_{L}(k+1)}.
\end{eqnarray*}
Note that the term $\prod_{k\in\{1,\ldots,n\}\backslash(U\cup(U-1))}(Y_{N}^{\epsilon_{k}})_{i_{L}(k),i_{L}(k+1)}$ is a product of the random variables $(\zeta_{j})_{j\in L}$ and their inverses, possibly with repetition.
Thus by Lemma \ref{P6}, fixing $i_{L}$ and summing over all $i_U$, we have
\[\sum_{i_U:U\to\{1,\ldots,N\}}\big|\mathbb{E}\prod_{k=1}^{n}(Y_{N}^{\epsilon_{k}})_{i(k),i(k+1)}\big|\leq N^{|U|-\dim\lspan\{\sum_{k\in S}v_{k}:S\in\pi_{1}\}}.\]
Summing now over all $i_{L}$ with $\ker i_{L}=\pi$, we obtain
\begin{equation}\label{ExpandBound}
\sum_{\{i_L:\ker i_{L}=\pi\}}\;\sum_{i_U:U\to\{1,\ldots,N\}}\big|\mathbb{E}\prod_{k=1}^{n}(Y_{N}^{\epsilon_{k}})_{i(k),i(k+1)}\big|\leq N^{|\pi|+|U|-\dim\lspan\{\sum_{k\in S}v_{k}:S\in\pi_{1}\}}.
\end{equation}
Let $\sim$ be the equivalence relation on $(U-1)\cup(U+1)$ generated by $l-1\sim l+1$ $\forall l\in U$. Let $\pi_{2}$ be the partition of $(U-1)\cup(U+1)$ that corresponds to $\sim$. By Lemma \ref{P11} and Lemma \ref{P7},
\[\dim\lspan\{\sum_{k\in S}v_{k}:S\in\pi_{1}\}\geq|\pi_{1}|-|\pi_{1}\vee\pi_{2}|.\]
Thus,
\begin{align}\label{PartitionEstimate}
|\pi|-\dim\lspan\{&\sum_{k\in S}v_{k}:S\in\pi_{1}\}\nonumber\leq|\pi|-|\pi_{1}|+|\pi_{1}\vee\pi_{2}|\nonumber\\
&=|\pi|-|\pi\upharpoonright_{(U-1)\cup(U+1)}|+|(\pi\upharpoonright_{(U-1)\cup(U+1)})\vee\pi_{2}|\nonumber\\
&=|\pi\vee(\pi_{2}\cup\{\{l\}:l\in L\backslash((U-1)\cup(U+1))\})|,
\end{align}
where the last equation follows from Lemma \ref{P8} by taking $K=(U-1)\cup(U+1)$ and $\lambda=\pi_{2}$.

\smallskip
\noindent
Case I: $|\pi\vee(\pi_{2}\cup\{\{l\}:l\in L\backslash((U-1)\cup(U+1))\})|\leq\frac{1}{2}(n+1-2|U|)$.

In this case, by \eqref{PartitionEstimate},
\[|\pi|-\dim\lspan\{\sum_{k\in S}v_{k}:S\in\pi_{1}\}\leq\frac{1}{2}(n+1-2|U|).\]
Thus, by \eqref{ExpandBound},
\[\sum_{\ker i_{L}=\pi}\;\sum_{i_U:U\to\{1,\ldots,N\}}\big|\mathbb{E}\prod_{k=1}^{n}(Y_{N}^{\epsilon_{k}})_{i(k),i(k+1)}\big|\leq N^{|U|+\frac{1}{2}(n+1-2|U|)}.\]
Since $X_{N}=\frac{1}{\sqrt{N}}Y_{N}$,
\[\sum_{\ker i_{L}=\pi}\;\sum_{i_U:U\to\{1,\ldots,N\}}\big|\mathbb{E}\prod_{k=1}^{n}(X_{N}^{\epsilon_{k}})_{i(k),i(k+1)}\big|\leq N^{|U|+\frac{1}{2}(n+1-2|U|)-\frac{n}{2}}=\sqrt{N}.\]
By \eqref{Estimate},
\[\sum_{\ker i_{L}=\pi}\;\sum_{i_U:U\to\{1,\ldots,N\}}\left|\mathbb{E}\prod_{S\in\sigma}\left(\prod_{k\in S}(X_{N}^{\epsilon_{k}})_{i(k),i(k+1)}-\mathbb{E}\left(\prod_{k\in S}(X_{N}^{\epsilon_{k}})_{i(k),i(k+1)}\right)\right)\right|\leq\sqrt{N}.\]

\smallskip
\noindent
Case II: $|\pi\vee(\pi_{2}\cup\{\{l\}:l\in L\backslash((U-1)\cup(U+1))\})|>\frac{1}{2}(n+1-2|U|)$.

By Lemma \ref{P12}, $|\pi_{2}|\leq|(U+1)\backslash(U-1)|$ so
\begin{align*}
|\pi_{2}\cup\{\{l\}:l\in L\backslash((U-1)&\cup(U+1))\}|
=|\pi_{2}|+|L\backslash((U-1)\cup(U+1))|\\
&\leq|(U+1)\backslash(U-1)|+|L|-|(U-1)\cup(U+1)|\\
&=|L|-|U-1|=|L|-|U|=n+1-2|U|.
\end{align*}
By Lemma \ref{P10}, $\pi_{2}\cup\{\{l\}:l\in L\backslash((U-1)\cup(U+1))\}$ contains a block $S_{1}$ such that
\begin{equation}\label{Singleton}
\{i_{L}(l):l\in S_{1}\}\cap\{i_{L}(l):l\in L\backslash S_{1}\}=\emptyset
\end{equation}
If $S_{1}\in\pi_{2}$ then by Lemma \ref{P14}\ref{it:P14.1},
\begin{equation}\label{LB}
S_{1}=L(S_{0})=\{k\in S_{0}:\epsilon_{k}=1\}\cup\{k+1:k\in S_{0}\text{ and }\epsilon_{k}=*\}
\end{equation}
for some $S_{0}\in\sigma$.
If $S_{1}=\{l\}$ for some $l\in L\backslash((U-1)\cup(U+1))$ and $l\neq 1,n+1$, then by Lemma \ref{P14}\ref{it:P14.2},
$S_{1}$ is also the form \eqref{LB}.
If $S_{1}=\{1\}$ or $\{n+1\}$ then since $1$ and $n+1$ are both in $L$ (by the definition of $U$), it follows from \eqref{Singleton} that $i_{L}(1)\neq i_{L}(n+1)$ and so
\begin{multline}\label{eq:sumE0}
\sum_{\ker i_{L}=\pi}\;\sum_{i_U:U\to\{1,\ldots,N\}}\delta_{i(n+1),i(1)}
\Bigg|\mathbb{E}\prod_{S\in\sigma}\Bigg(\prod_{k\in S}(X_{N}^{\epsilon_{k}})_{i(k),i(k+1)} \\
-\mathbb{E}\left(\prod_{k\in S}(X_{N}^{\epsilon_{k}})_{i(k),i(k+1)}\right)\Bigg)\Bigg|=0.
\end{multline}

If $S_{1}$ is of the form \eqref{LB} then by \eqref{Singleton},
\[\{i_{L}(l):l\in L(S_{0})\}\cap\{i_{L}(l):l\in L\backslash L(S_{0})\}=\emptyset.\]
By Lemma \ref{P13}, $\displaystyle\cup_{S\neq S_{0}}L(S)\subset L\backslash L(S_{0})$. So
\[\{i_{L}(l):l\in L(S_{0})\}\cap\{i_{L}(l):l\in\cup_{S\neq S_{0}}L(S)\}=\emptyset.\]
Note that for each $S\in\sigma$, $\prod_{k\in S}(X_{N}^{\epsilon_{k}})_{i(k),i(k+1)}$ depends only on $\{\zeta_{i_{L}(l)}:l\in L(S)\}$.  Thus, the random variable
\[\prod_{k\in S_{0}}(X_{N}^{\epsilon_{k}})_{i(k),i(k+1)}\]
is independent of the random variables
\[\prod_{k\in S}(X_{N}^{\epsilon_{k}})_{i(k),i(k+1)},\quad S\neq S_{0}.\]
Thus,
\[\left|
\mathbb{E}\prod_{S\in\sigma}\left(\prod_{k\in S}(X_{N}^{\epsilon_{k}})_{i(k),i(k+1)}-\mathbb{E}\left(\prod_{k\in S}(X_{N}^{\epsilon_{k}})_{i(k),i(k+1)}\right)\right)\right|=0.\]
So again~\eqref{eq:sumE0} holds.

Combining the conclusions of Case I and Case II and summing over all partitions $\pi$ of $L$, we get
\begin{multline*}
\sum_{i_{L}:L\to\{1,\ldots,N\}}\sum_{i_U:U\to\{1,\ldots,N\}}\delta_{i(n+1),i(1)}\Bigg|
\mathbb{E}\prod_{S\in\sigma}\Bigg(\prod_{k\in S}(X_{N}^{\epsilon_{k}})_{i(k),i(k+1)} \\
-\mathbb{E}\left(\prod_{k\in S}(X_{N}^{\epsilon_{k}})_{i(k),i(k+1)}\right)\Bigg)\Bigg|\leq C\sqrt{N},
\end{multline*}
where $C$ is the number of partitions of $L$.
By \eqref{Onealign}, the result follows.
\end{proof}
\begin{proof}[Proof of Proposition \ref{OneProp}]
Let
\[Z_{N,I}=\prod_{k\in I}d_{k}X_{N}^{\epsilon_{k}}.\]
By Lemma \ref{Offdiag} and Lemma \ref{P15}, for every integer $p\geq 1$,
\[|Z_{N,I}-\mathbb{E}Z_{N,I}|_{2p}\leq C\]
and
\[|\mathbb{E}Z_{N,I}-\mathbb{E}\circ\mathrm{diag}Z_{N,I}|_{2p}\leq C\]
where $C$ depends only on $n$ and $p$. So by Lemma \ref{P19},
\begin{multline*}
\left|\mathbb{E}\circ\mathrm{tr}\prod_{I\in\sigma}(Z_{N,I}-\mathbb{E}Z_{N,I})
-\mathbb{E}\circ\mathrm{tr}\prod_{I\in\sigma}(Z_{N,I}-\mathbb{E}\circ\mathrm{diag}Z_{N,I})\right| \\[1ex]
\leq C\sup_{I\in\sigma}|\mathbb{E}Z_{N,I}-\mathbb{E}\circ\mathrm{diag}Z_{N,I}|_{2},
\end{multline*}
where $C$ depends only on $n$. By Lemma \ref{Offdiag} for $p=1$,
\[|\mathbb{E}Z_{N,I}-\mathbb{E}\circ\mathrm{diag}Z_{N,I}|_{2}\leq\frac{C}{\sqrt{N}},\quad I\in\sigma.\]
Therefore,
\[\left|\mathbb{E}\circ\mathrm{tr}\prod_{I\in\sigma}(Z_{N,I}-\mathbb{E}Z_{N,I})-\mathbb{E}\circ\mathrm{tr}\prod_{I\in\sigma}(Z_{N,I}-\mathbb{E}\circ\mathrm{diag}Z_{N,I})\right|\leq\frac{C}{\sqrt{N}}.\]
Thus, by Lemma \ref{OneLem}, the result follows.
\end{proof}

We are now ready to prove the main result.
For a C$^*$-algebra $B$,
by $B\langle X,X^*\rangle$ we denote the $*$-algebra of polynomials in noncommuting variables $X$ and $X^*$ with coefficients on $B$;
technically this is the algebraic free product of the three algebras $B$, $\Cpx[X]$ and $\Cpx[X^*]$ with amalgamation over the scalars.
We endow $B\langle X,X^*\rangle$ with the obvious $*$-operation.
\begin{thm}\label{thm:XRDiag}
Consider the C$^*$-algebra $B=C[0,1]$ with tracial state $\tau:B\to\Cpx$ obtained by integration using Lebesgue measure.
Let $\Ec:B\langle X,X^*\rangle\to B$ be the linear, self-adjoint, $B$-bimodular
map that is the identity on $B$ and so that with respect to $\Ec$, $X$ is $B$-valued
R-diagonal with even alternating moments given by, for every $n\in\Nats$ and $b_1,\ldots,b_{2n}\in B$,
\begin{align}
&\Ec(b_{1}X^*b_{2}Xb_{3}X^*b_{4}X\cdots b_{2n-1}X^*b_{2n}X) \label{eq:EX*X} \\
&\hspace{10em}
=\sum_{\pi\in\Pc(n)}b_1\Lambda_\pi(b_3,b_5,\ldots,b_{2n-1})\tau\big(\Gamma_\pi(b_2,b_4,\ldots,b_{2n})\big) \notag \\
&\Ec(b_{1}Xb_{2}X^*b_{3}Xb_{4}X^*\cdots b_{2n-1}Xb_{2n}X^*) \label{eq:EXX*} \\
&\hspace{10em}
=\sum_{\pi\in\Pc(n)}\Gamma_\pi(b_1,b_3,\ldots,b_{2n-1})\tau\big(\Lambda_\pi(b_2,b_4,\ldots,b_{2n-2})b_{2n}\big). \notag 
\end{align}
Then for all $n\in\Nats$, $\eps(1),\ldots,\eps(n)\in\{1,*\}$ and all $b_1,\ldots,b_n\in B$,
we have
\begin{equation}\label{eq:EXeps}
\lim_{N\to\infty}\Eb\circ\tr(D_N(b_1)X_N^{\eps(1)}D_N(b_2)X_N^{\eps(2)}\cdots D_N(b_n)X_N^{\eps(n)})
=\tau\circ\Ec(b_1X^{\eps(1)}b_2X^{\eps(2)}\cdots b_nX^{\eps(n)}),
\end{equation}
where, 
for $b\in B$, $D_N(b)$ is the scalar diagonal matrix
\[
D_N(b)=\diag\big(b(\frac1N),b(\frac2N),\ldots,b(\frac NN)\big).
\]
\end{thm}
\begin{proof}
Let $n\in\Nats$
and suppose $b_1,\ldots,b_n\in B$ and
$\eps(1),\ldots,\eps(n)\in\{1,*\}$ are arbitrary.
We will prove~\eqref{eq:EXeps} by induction on $n$.
In the case of $n=1$, the right-hand-side of~\eqref{eq:EXeps} is zero and, by Lemma~\ref{Odd}, so is
the left-hand-side of~\eqref{eq:EXeps}.
For the induction step, let $\sigma=\sigma(\eps)$ be the maximal alternating interval partition of $\eps$
(see Definition~\ref{def:maxaltptn}).
For $I\in\sigma$, let
\[
c_I=\Ec\left(\prod_{j\in I}b_jX^{\eps(j)}\right)\in B,
\]
where the product is taken in increasing order of the index $j$.
By $B$-valued R-diagonality of $X$,
\[
\tau\circ\Ec\left(\prod_{I\in\sigma}\left(\prod_{j\in I}b_jX^{\eps(j)}-c_I\right)\right)=0,
\]
where the product over $I\in\sigma$ is taken in order of increasing elements of the interval blocks $I$
(since $\sigma$ is an interval partition, given two distinct blocks,
all the elements of one of them are less than all the elements of the other).
Expanding the above product over $I\in\sigma$, we get a sum of $2^{|\sigma|}$ terms that enables
$\tau\circ\Ec(\prod_{j=1}^nb_jX^{\eps(j)})$ to be expressed as $(-1)^{|\sigma|-1}\tau(\prod_{I\in\sigma}c_I)$
plus the sum of $2^{|\sigma|}-2$ terms, each of the form
\begin{equation}\label{eq:EfX}
(-1)^{|\sigma\backslash\sigma'|-1}\tau\circ\Ec\left(\prod_{j\in K}f_jX^{\eps(j)}\right),
\end{equation}
where $K$ is the union of a proper subset $\sigma'$ of $\sigma$ and for certain $f_j\in B$,
equal to the product of $b_j$ and some of $(c_I)_{I\in\sigma\backslash\sigma'}$.

We will show
\begin{equation}\label{eq:withDcI}
\lim_{N\to\infty}\mathbb{E}\circ\mathrm{tr}\left(\prod_{I\in\sigma}\left(\prod_{j\in I}D_N(b_{j})X_{N}^{\epsilon(j)}
-D_N(c_I)\right)\right)=0.
\end{equation}
This will prove the induction step, because expansion of the left-hand-side of~\eqref{eq:withDcI} as a sum of $2^{|\sigma|}$ terms
will enable
\[
\lim_{N\to\infty}\mathbb{E}\circ\mathrm{tr}\left(\prod_{j=1}^nD_N(b_j)X_N^{\eps(j)}\right)
\]
to be written as
\[
(-1)^{|\sigma|-1}\lim_{N\to\infty}\Eb\circ\tr\left(\prod_{I\in\sigma}D_N(c_I)\right)=(-1)^{|\sigma|-1}\tau\left(\prod_{I\in\sigma}c_I\right)
\]
plus the sum of $2^{|\sigma|-1}$ terms,
each equal to
\begin{equation}\label{eq:limEfXN}
(-1)^{|\sigma\backslash\sigma'|-1}\lim_{N\to\infty}\Eb\circ\tr\left(\prod_{j\in K}D_N(f_j)X_N^{\eps(j)}\right),
\end{equation}
for the same $K$ and $f_j$ as appeared in~\eqref{eq:EfX}.
By the inductive hypothesis, each of the terms in~\eqref{eq:limEfXN} is equal to the corresponding term in~\eqref{eq:EfX}.
This shows that proof of the induction step will follow, once we have proved~\eqref{eq:withDcI}.

In order to verify~\eqref{eq:withDcI}, we will use
Proposition~\ref{OneProp}, which yields
\begin{equation}\label{eq:3.1result}
\lim_{N\to\infty}\mathbb{E}\circ\mathrm{tr}\left(\prod_{I\in\sigma}\left(\prod_{j\in I}D_N(b_{j})X_{N}^{\epsilon(j)}
-\mathbb{E}\circ\mathrm{diag}\left(\prod_{j\in I}D_N(b_{j})X_{N}^{\epsilon(j)}\right)\right)\right)=0.
\end{equation}
For $I\in\sigma$, if $n$ is even, then
from Proposition~\ref{prop:diagexpect} and \eqref{eq:EX*X}--\eqref{eq:EXX*}, we have
\begin{equation}\label{eq:alternatingdiag}
\lim_{N\to\infty}\left\|\Eb\circ\diag\left(\prod_{j\in I}D_N(b_{j})X_{N}^{\epsilon(j)}\right)-D_N(c_I)\right\|=0.
\end{equation}
while if $n$ is odd, then by R-diagonality of $X$ we have $c_I=0$ and from Lemma~\ref{Odd}, we see that also
in this case~\eqref{eq:alternatingdiag} holds.
We now write, for each $I\in\sigma$,
\begin{multline*}
\prod_{j\in I}D_N(b_{j})X_{N}^{\epsilon(j)}-D_N(c_I)
=\left(\prod_{j\in I}D_N(b_{j})X_{N}^{\epsilon(j)}
-\mathbb{E}\circ\mathrm{diag}\left(\prod_{j\in I}D_N(b_{j})X_{N}^{\epsilon(j)}\right)\right) \\
+\left(\mathbb{E}\circ\mathrm{diag}\left(\prod_{j\in I}D_N(b_{j})X_{N}^{\epsilon(j)}\right)-D_N(c_I)\right)
\end{multline*}
and, in the left-hand-side of~\eqref{eq:withDcI}, distribute, resulting in a sum of $2^{|\sigma|}$ limits, each of which will
be seen to equal $0$.
That the first of these limits is zero is precisely the import of~\eqref{eq:3.1result}.
That each of the other limits is zero is a consquence of~\eqref{eq:alternatingdiag} and H\"older's inequality, (see, Lemma~\ref{P18}).
Indeed, each of the other limits is of the form
\begin{equation}\label{eq:prodFI}
\lim_{N\to\infty}\Eb\circ\tr\left(\prod_{I\in\sigma}F_I\right),
\end{equation}
where $F_I$ is either
\[
\left(\prod_{j\in I}D_N(b_{j})X_{N}^{\epsilon(j)}
-\mathbb{E}\circ\mathrm{diag}\left(\prod_{j\in I}D_N(b_{j})X_{N}^{\epsilon(j)}\right)\right)
\]
or
\[
\left(\mathbb{E}\circ\mathrm{diag}\left(\prod_{j\in I}D_N(b_{j})X_{N}^{\epsilon(j)}\right)-D_N(c_I)\right)
\]
and for at least one $I\in\sigma$ it is the latter.
Now from~\eqref{eq:alternatingdiag}, we conclude that, for every $I\in\sigma$,
\[
\left\|\mathbb{E}\circ\mathrm{diag}\left(\prod_{j\in I}D_N(b_{j})X_{N}^{\epsilon(j)}\right)\right\|
\]
remains bounded as $N\to\infty$.
From Lemma~\ref{Offdiag}, we have that, for every $I\in\sigma$ and every integer $p\ge1$,
\[
\left|\prod_{j\in I}D_N(b_{j})X_{N}^{\epsilon(j)}\right|_{2p}
\]
remains bounded as $N\to\infty$.
Consequently, for every $I\in\sigma$,
\[
\left|\prod_{j\in I}D_N(b_{j})X_{N}^{\epsilon(j)}-\mathbb{E}\circ\mathrm{diag}\left(\prod_{j\in I}D_N(b_{j})X_{N}^{\epsilon(j)}\right)\right|_{2p}
\]
remains bounded as $N\to\infty$.
Of course, from~\eqref{eq:alternatingdiag} we get, for every $I$ and $p$,
\[
\lim_{N\to\infty}\left|\mathbb{E}\circ\mathrm{diag}\left(\prod_{j\in I}D_N(b_{j})X_{N}^{\epsilon(j)}\right)-D_N(c_I)\right|_{2p}=0.
\]
Consequently, taking $p=|\sigma|$ and applying H\"older's inequality, we get that for every product $\prod_{I\in\sigma}F_I$
of the form described at~\eqref{eq:prodFI},
\[
\lim_{N\to\infty}\left|\prod_{I\in\sigma}F_I\right|_2=0.
\]
Now using the Cauchy--Schwarz inequality, we conclude
\[
\lim_{N\to\infty}\Eb\circ\tr\left(\prod_{I\in\sigma}F_I\right)=0.
\]
This finishes the proof of~\eqref{eq:withDcI}, and of the theorem.
\end{proof}

\section{Calculating $\Lambda_\pi$ and certain moments and cumulants}
\label{sec:calcs}

Here are some results that will allow us to calculate $\Lambda_\pi$ for many partitions $\pi$.
The first is an easy calculation:
\begin{lemma}\label{lem:1n}
Suppose $n\ge2$ and $\pi=1_n$ is the partition of $\{1,\ldots,n\}$ into one block.
Then
\[
\Lambda_\pi(g_1,\ldots,g_{n-1})=\prod_{p=1}^{n-1}\tau(g_p)
\]
is constant.
\end{lemma}
\begin{proof}
We have $I_\pi(p)=\{1,2,\ldots,p\}$ for all $p\in\{1,\ldots,n-1\}$ and, thus,
\[
E(\pi,t)=\{(t_1,\ldots,t_{n-1})\in\Reals^{n-1}\mid \forall p\in\{1,\ldots,n-1\},\,0\le t+\sum_{j=1}^p t_j\le1\}.
\]
The change of variables
\[
s_p=t+t_p+\sum_{j=1}^{p-1}t_j,\qquad(1\le p\le n-1)
\]
preserves Lebesgue measure and sends $E(\pi,t)$ onto $[0,1]^{n-1}$ and we get
\begin{multline*}
\Lambda_{1_n}(g_1,\ldots,g_{n-1})(t)
=\int_{E(\pi,t)}\prod_{p=1}^{n-1}g_p(t+\sum_{j=1}^pt_j)\,d\lambda((t_j)_{j=1}^{n-1}) \\
=\int_{[0,1]^{n-1}}\bigg(\prod_{n=1}^{n-1}g_p(s_p)\bigg)\,d\lambda((s_j)_{j=1}^{n-1})
=\prod_{p=1}^{n-1}\bigg(\int_0^1g_p(s)\,ds\bigg)
=\prod_{p=1}^{n-1}\tau(g_p).
\end{multline*}
\end{proof}

The next lemma concerns partitions obtained by rotations of the underlying set.
Let $\ell=\ell_n:\{1,\ldots,n\}\to\{1,\ldots,n\}$ denote the left rotation map: $\ell(j)=j-1\mod n$.
For $\pi\in\Pc(n)$, let $\ell(\pi)=\{\ell(S)\mid S\in\pi\}$ denote the partition obtained by rotating the underlying set according to $\ell$.

\begin{lemma}\label{lem:rotate}
For every $n\in\Nats$ and $\pi\in\Pc(n)$,
\begin{equation}\label{eq:tauLamL}
\tau(\Lambda_\pi(g_1,\ldots,g_{n-1})g_n)=\tau(\Lambda_{\ell(\pi)}(g_2,\ldots,g_n)g_1).
\end{equation}
\end{lemma}
\begin{proof}
Let
\[
E(\pi)=\bigcup_{t\in[0,1]}E(\pi,t)\times\{t\},\qquad
E(\ell(\pi))=\bigcup_{u\in[0,1]}E(\ell(\pi),u)\times\{u\}.
\]
Recalling that $I_\pi(n)=\emptyset$, we have
\begin{align*}
\tau(\Lambda_\pi(g_1,\ldots,g_{n-1})g_n)&=\int_{E(\pi)}\bigg(\prod_{p=1}^ng_p\bigg(t+\sum_{j\in I_\pi(p)}t_j\bigg)\bigg)
\,d\lambda\big((t_j)_{j\in J_\pi},t\big), \\
\tau(g_1\Lambda_{\ell(\pi)}(g_2,\ldots,g_n))&=\int_{E(\ell(\pi))}\bigg(\prod_{p=1}^ng_{\ell^{-1}(p)}\bigg(u+\sum_{j\in I_{\ell(\pi)}(p)}u_j\bigg)\bigg)
\,d\lambda\big((u_j)_{j\in J_{\ell(\pi)}},u\big) \\
&=\int_{E(\ell(\pi))}\bigg(\prod_{p=1}^ng_p\bigg(u+\sum_{j\in I_{\ell(\pi)}(\ell(p))}u_j\bigg)\bigg)
\,d\lambda\big((u_j)_{j\in J_{\ell(\pi)}},u\big).
\end{align*}
We will show that there is linear isomorphism
$\Theta:\Reals^{J_\pi}\times\Reals\to\Reals^{J_{\ell(\pi)}}\times\Reals$
that preserves Lebesgue measure and satisfies
$\Theta(E(\pi))=E(\ell(\pi))$ and that if
\begin{equation}\label{eq:Ttu}
\Theta((t_j)_{j\in J_\pi},t)=((u_j)_{j\in J_{\ell(\pi)}},u),
\end{equation}
then
\begin{equation}\label{eq:tusums}
t+\sum_{j\in I_\pi(p)}t_j=u+\sum_{j\in I_{\ell(\pi)}(\ell(p))}u_j\qquad(1\le p\le n).
\end{equation}
This will yield the desired identity~\eqref{eq:tauLamL}, after performing a change of variables of integration.

From Lemma~\ref{lem:LambdapiAlt}, we have the isomorphisms
\begin{align*}
\Phi_\pi&:\Reals^{J_\pi}\times\Reals\to K_\pi\subseteq\Reals^n \\
\Phi_{\ell(\pi)}&:\Reals^{J_{\ell(\pi)}}\times\Reals\to K_{\ell(\pi)}\subseteq\Reals^n.
\end{align*}
The cyclic permutation map $C:\Reals^n\to\Reals^n$ given by $C(s_1,\ldots,s_n)=(s_2,s_3,\ldots,s_n,s_1)$ is an isomorphism that sends $K_\pi$
onto $K_{\ell(\pi)}$.
We let 
\[
\Theta=\Phi_{\ell(\pi)}^{-1}\circ C\circ\Phi_\pi:\Reals^{J_\pi}\times\Reals\to\Reals^{J_{\ell(\pi)}}\times\Reals
\]
and from the definitions of $\Phi_\pi$ and $\Phi_{\ell(\pi)}$, we immediately see that~\eqref{eq:Ttu}
implies~\eqref{eq:tusums}.
From this, we deduce $\Theta(E(\pi))=E(\ell(\pi))$.
It remains only to see that $\Theta$ preserves Lebesque measure.

\smallskip\noindent
{\em Case 1: $\{1\}\in\pi$.}
Then $1\notin J_\pi$ and $J_{\ell(\pi)}=\{j-1\mid j\in J_\pi\}$ and
\[
I_{\ell(\pi)}(p-1)=\{j-1\mid j\in I_\pi(p)\},\quad(2\le p\le n).
\]
Thus, from~\eqref{eq:tusums}, we have
\[
t+\sum_{j\in I_\pi(p)}t_j=u+\sum_{j\in I_\pi(p)}u_{j-1},\quad(2\le p\le n)
\]
and, since $I_\pi(1)=\emptyset=I_\pi(n)$, taking $p=1$ in~\eqref{eq:tusums}, we also get $t=u$.
Thus, we get $t_j=u_{j-1}$ for all $j\in J_\pi$ and we see that the mapping $\Theta$ amounts to a relabelling of the variables,
which preserves Lebesgue measure.

\smallskip\noindent
{\em Case 2: $\{1\}\notin\pi$.}
Then $1\in J_\pi$.
Recall that $S_\pi(1)$ denotes the block of $\pi$ that contains $1$.
Let $m=\max S_\pi(1)$.
Then $m\notin J_\pi$ and we have
\[
J_{\ell(\pi)}=\{j-1\mid j\in J_\pi\setminus\{1\}\}\cup\{m-1\}
\]
and
\[
I_{\ell(\pi)}(p-1)=\begin{cases}
\{j-1\mid j\in I_\pi(p)\setminus\{1\}\},&2\le p<m \\
\{j-1\mid j\in(I_\pi(p)\cup S_\pi(1))\setminus\{1\}\},&m\le p\le n.
\end{cases}
\]
Thus, noting that $I_\pi(p)\cap S_\pi(1)=\emptyset$ whenever $m\le p\le n$, from~\eqref{eq:tusums}, we have
\begin{equation}\label{eq:tusumsdetail}
t+\sum_{j\in I_\pi(p)}t_j=\begin{cases}
u+\sum_{j\in I_\pi(p)\setminus\{1\}}u_{j-1},&2\le p<m \\
u+\sum_{j\in I_\pi(p)}u_{j-1}+\sum_{j\in S_\pi(1)\setminus\{1\}}u_{j-1},&m\le p\le n. \\
\end{cases}
\end{equation}
Take $p\in J_\pi\setminus\{1\}$.
Then we have $I_\pi(p)=I_\pi(p-1)\cup\{p\}$ and, consequently,
we find (keeping in mind $m\notin J_\pi$)
\[
t_p=\bigg(t+\sum_{j\in I_\pi(p)}t_j\bigg)-\bigg(t+\sum_{j\in I_\pi(p-1)}t_j\bigg)=u_{p-1},\quad(p\in J_\pi\setminus\{1\}).
\]
On the other hand, taking $p=n$, since $I_\pi(n)=\emptyset$, from~\eqref{eq:tusumsdetail}, we get
\[
t=u+\sum_{j\in S_\pi(1)\setminus\{1\}}u_{j-1}.
\]
Since $I_\pi(1)=\{1\}$ and $I_{\ell(\pi)}(n)=\emptyset$, from~\eqref{eq:tusums}, we get $t+t_1=u$.
Thus, we have
\[
t_1=(t_1+t)-t=-\sum_{j\in S_\pi(1)\setminus\{1\}}u_{j-1}.
\]
Thus, writing $J_\pi=\{j(1),j(2),\ldots,j(n-|\pi|)\}$ with $1=j(1)<j(2)<\cdots<j(n-|\pi|)$, we have
\[
\left(\begin{matrix}
t_{j(2)} \\
t_{j(3)} \\
\vdots \\
t_{j(n-|\pi|)} \\
t_1 \\
t
\end{matrix}\right)
=A
\left(\begin{matrix}
u_{j(2)-1} \\
u_{j(3)-1} \\
\vdots \\
u_{j(n-|\pi|)-1} \\
u_{m-1} \\
u
\end{matrix}\right)
\]
where $A$ is a lower triangular matrix whose diagonal entries form the list $(1,1,\ldots,1,-1,1)$.
Thus, the change of variables implimented by $\Theta$ preserves Lebesgue measure, as required.
\end{proof}

The next lemma handles the case when $\pi$ splits along two adjacent intervals.
Given integers $1\le x<n$ and given $\pi_1\in\Pc(x)$, $\pi_2\in\Pc(n-x)$, let us write
\[
\pi=\pi_1\oplus\pi_2
\]
for the partition $\pi\in\Pc(n)$ given by
$\pi=\pi_1\cup\pit_2$, where $\pit_2$ is obtained by translating $\pi_2$ distance $x$ to the right, namely,
\[
\pit_2=\big\{\{x+j\mid j\in S\}\mid S\in\pi_2\big\}.
\]
\begin{lemma}\label{lem:twointervals}
Given integers $1\le x <n$ and letting $\pi=\pi_1\oplus\pi_2\in\Pc(n)$ for some $\pi_1\in\Pc(x)$ and $\pi_2\in\Pc(n-x)$, we have
\[
\Lambda_\pi(g_1,\ldots,g_{n-1})=\Lambda_{\pi_1}(g_1,\ldots,g_{x-1})g_x\Lambda_{\pi_2}(g_{x+1},\ldots,g_{n-1}).
\]
\end{lemma}
\begin{proof}
We have
\[
J_\pi=J_{\pi_1}\cup\{j+x\mid j\in J_{\pi_2}\}
\]
and
\[
I_\pi(p)=\begin{cases}
I_{\pi_1}(p),&1\le p\le x \\
\{j+x\mid j\in I_{\pi_2}(p-x)\},&x<p\le n.
\end{cases}
\]
In particular $I_\pi(x)=\emptyset$.
Thus, for every $t\in[0,1]$,
\begin{align*}
E(\pi,t)&=\bigg\{\big((t_j)_{j\in J_{\pi_1}},(t_{j+x})_{j\in J_{\pi_2}}\big)\;\bigg|\;
\begin{aligned}[t]
&\forall 1\le p<x,\quad0<t+\sum_{j\in I_{\pi_1}(p)}t_j\le1, \\
&\forall x+1\le p<n,\quad0<t+\sum_{j\in I_{\pi_2}(p-x)}t_{x+j}\le1\bigg\}
\end{aligned} \\
&=E(\pi_1,t)\times E(\pi_2,t)
\end{align*}
and
\begin{align*}
\Lambda_\pi(g_1,\ldots,g_{n-1})
&=\int_{E(\pi,t)}\bigg(\prod_{p=1}^{x-1}g_p\bigg(t+\sum_{j\in I_{\pi_1}(p)}t_j\bigg)\bigg)\,g_x(t) \\
&\qquad\qquad\cdot\bigg(\prod_{p=1}^{n-x-1}g_{x+p}\bigg(t+\sum_{j\in I_{\pi_2}(p)}t_{x+j}\bigg)\bigg)\,d\lambda((t_j)_{j\in J_{\pi_1}})\,d\lambda((t_{x+j})_{j\in J_{\pi_2}}) \\
&=\Lambda_{\pi_1}(g_1,\ldots,g_{x-1})(t)\,g_x(t)\,\Lambda_{\pi_2}(g_{x+1},\ldots,g_{n-1})(t).
\end{align*}
\end{proof}

\begin{lemma}\label{lem:internalinterval}
Suppose $\pi\in\Pc(n)$ and $\pi=\pit_1\cup\pit_2$, where $\pit_1$ is a partition of
$S_1=\{1,\ldots,x\}\cup\{x+y+1,\ldots,n\}$
and $\pit_2$ is a partition of $S_2=\{x+1,\ldots,x+y\}$,
for some integers $1\le x<x+y\le n-1$.
Let $\pi_1\in\Pc(n-y)$ and $\pi_2\in\Pc(y)$ be the partitions obtained from $\pit_1$
and $\pit_2$ by applying the order-preserving bijections from $S_1$ onto $\{1,\ldots,n-y\}$ and from $S_2$ onto
and $\{1,\ldots,y\}$, respectively.
Then
\[
\Lambda_\pi(g_1,\ldots,g_{n-1})
=\Lambda_{\pi_1}(g_1,\ldots,g_{x-1},g_x\Lambda_{\pi_2}(g_{x+1},\ldots,g_{x+y-1})g_{x+y},g_{x+y+1},\ldots,g_{n-1}).
\]
\end{lemma}
\begin{proof}
Let $g_n\in C[0,1]$.
It will suffice to show
\begin{multline*}
\tau\big(\Lambda_\pi(g_1,\ldots,g_{n-1})g_n\big) \\
=\tau\big(\Lambda_{\pi_1}(g_1,\ldots,g_{x-1},g_x\Lambda_{\pi_2}(g_{x+1},\ldots,g_{x+y-1})g_{x+y},g_{x+y+1},\ldots,g_{n-1})g_n\big).
\end{multline*}
Applying Lemma~\ref{lem:rotate} $x$ times in succession, we get
\begin{equation}\label{eq:tauLnx}
\tau\big(\Lambda_\pi(g_1,\ldots,g_{n-1})g_n\big)
=\tau\big(\Lambda_{\ell_n^x(\pi)}(g_{x+1},g_{x+2},\ldots,g_n,g_1,g_2,\ldots,g_{x-1})g_x\big).
\end{equation}
The partition $\ell_n^x(\pi)$ obtained by rotating $\pi$ a total of $x$ times to the left is split by the invervals $\ell_n^x(S_2)=\{1,\ldots,y\}$
and $\ell_n^x(S_1)=\{y+1,y+2,\ldots,n\}$
and, in the notation introduced above Lemma~\ref{lem:twointervals},
\[
\ell_n^x(\pi)=\pi_2\oplus \ell_{n-y}^x(\pi_1).
\]
Applying Lemma~\ref{lem:twointervals}, we have
\begin{multline*}
\Lambda_{\ell_n^x(\pi)}(g_{x+1},g_{x+2},\ldots,g_n,g_1,g_2,\ldots,g_{x-1}) \\
=\Lambda_{\pi_2}(g_{x+1},\ldots,g_{x+y-1})g_{x+y}\Lambda_{\ell_{n-y}^x(\pi_1)}(g_{x+y+1},\ldots,g_n,g_1,\ldots,g_{x-1}).
\end{multline*}
Substituting into~\eqref{eq:tauLnx} and applying Lemma~\ref{lem:rotate} again $x$ times, we get
\begin{multline*}
\tau\big(\Lambda_\pi(g_1,\ldots,g_{n-1})g_n\big) \\
\begin{aligned}
&=\tau\big(\Lambda_{\ell_{n-y}^x(\pi_1)}(g_{x+y+1},\ldots,g_n,g_1,\ldots,g_{x-1})g_x\Lambda_{\pi_2}(g_{x+1},\ldots,g_{x+y-1})g_{x+y}\big) \\
&=\tau\big(\Lambda_{\pi_1}(g_1,\ldots,g_{x-1},g_x\Lambda_{\pi_2}(g_{x+1},\ldots,g_{x+y-1})g_{x+y},g_{x+y+1},\ldots,g_{n-1})g_n\big),
\end{aligned}
\end{multline*}
as required.
\end{proof}

The next lemma treats the case when a partition has two adjacent elements in the same block.
\begin{lemma}\label{lem:adjacent}
Let $\pi\in\Pc(n)$ and suppose $\{k,k+1\}\subseteq S\in\pi$ for some $k\in\{1,\ldots,n-1\}$.
Let $\pit\in\Pc(n-1)$ be obtained from $\pi$ by gluing $k$ and $k+1$ together;
namely, letting 
\[
F:\{1,\ldots,n\}\backslash\{k+1\}\to\{1,\ldots,n-1\}
\]
be the order-preserving bijection, we have
\[
\pit=\{F(S\backslash\{k+1\})\mid S\in\pi\}.
\]
Then
\[
\Lambda_\pi(g_1,\ldots,g_{n-1})=\Lambda_{\pit}(g_1,g_2,\ldots,\widehat{g_k},\ldots,g_{n-1})\tau(g_k),
\]
where, as usual, $\widehat{g_k}$ indicates that $g_k$ has been removed from the list of arguments, while all the others remain.
\end{lemma}
\begin{proof}
First, suppose $k=1$.
Let $m=\max S_\pi(1)$. 
Then $m\ge2$.
If $m=2$, then letting $\pi_2\in\Pc(n-2)$ be obtained by restricting $\pi$ to $\{3,\ldots,n\}$ and translating left by $2$, we have
\[
\pi=1_2\oplus\pi_2,\qquad\pit=1_1\oplus\pi_2.
\]
Applying Lemmas~\ref{lem:twointervals} and~\ref{lem:1n}, we have
\begin{align*}
\Lambda_\pi(g_1,\ldots,g_{n-1})
&=\Lambda_{1_2}(g_1)g_2\Lambda_{\pi_2}(g_3,\ldots,g_{n-1})=\tau(g_1)g_2\Lambda_{\pi_2}(g_3,\ldots,g_{n-1}) \\
&=\tau(g_1)\Lambda_{1_1}()\,g_2\Lambda_{\pi_2}(g_3,\ldots,g_{n-1})=\tau(g_1)\Lambda_{\pit}(g_2,\ldots,g_{n-1}).
\end{align*}
Now suppose $m>2$.
Then $\{1,2\}\subseteq J_\pi$ and
\[
J_\pit=\{j-1\mid j\in J_\pi\setminus\{1\}\}
\]
and
\[
I_\pi(p)=\begin{cases}
\{1\},&p=1 \\
\{1,2\}&p=2 \\
\{1,2\}\cup\{j+1\mid j\in I_\pit(p-1)\setminus\{1\}\},& 3\le p<m \\
\{j+1\mid j\in I_\pit(p-1)\},& m\le p<n.
\end{cases}
\]
Moreover, $I_\pi(p)\cap\{1,2\}=\emptyset$ whenever $p\ge m$
and $1\in I_\pit(p-1)$ if and only if $2\le p<m$.
Thus,
\[
E(\pi,t)=\bigg\{\big(t_1,t_2,(t_{j+1})_{j\in J_\pit\setminus\{1\}}\big)\;\bigg|\;
\begin{aligned}[t]
&0<t+t_1\le1,\;0<t+t_1+t_2\le1, \\[1ex]
&\forall p\in\{3,\ldots,m-1\}, \\
&\qquad0<t+t_1+t_2+\sum_{j\in I_\pit(p-1)\setminus\{1\}}t_{j+1}\le1, \\
&\forall p\in\{m,\ldots,n-1\},\;0<t+\sum_{j\in I_\pit(p-1)}t_{j+1}\le1\bigg\},
\end{aligned}
\]
whereas
\[
E(\pit,t)=\bigg\{\big(u_1,(u_j)_{j\in J_\pit\setminus\{1\}}\big)\;\bigg|\;
\begin{aligned}[t]
&0<t+u_1\le1, \\[1ex]
&\forall p\in\{3,\ldots,m-1\}, \\
&\qquad0<t+u_1+\sum_{j\in I_\pit(p-1)\setminus\{1\}}u_j\le1, \\
&\forall p\in\{m,\ldots,n-1\},\;0<t+\sum_{j\in I_\pit(p-1)}u_j\le1\bigg\}.
\end{aligned}
\]
The affine mapping $\Reals^{J_\pi}\to\Reals\times\Reals^{J_\pit}$ given by
\[
\big(t_1,t_2,(t_{j+1})_{j\in J_\pit\setminus\{1\}}\big)\mapsto\big(t+t_1,\big(t_1+t_2,(t_j)_{j\in J_\pit\setminus\{1\}}\big)\big)
\]
preserves Lebesgue measure and maps $E(\pi,t)$ onto $(0,1]\times E(\pit,t)$.
Thus, making the change of variables
\[
r=t+t_1,\qquad u_1=t_1+t_2,\qquad u_j=t_{j+1},\quad(j\in J_\pit\setminus\{1\}),
\]
we have
\begin{align*}
\Lambda_\pi(g_1,\ldots,g_{n-1})&=\int_{E(\pi,t)}
\begin{aligned}[t]
&g_1(t+t_1)g_2(t+t_1+t_2) \\
&\quad\cdot\bigg(\prod_{p=3}^{m-1}g_p\bigg(t+t_1+t_2+\sum_{j\in J_\pit(p-1)\setminus\{1\}}t_{j+1}\bigg)\bigg) \\
&\quad\cdot\bigg(\prod_{p=m}^{n-1}g_p\bigg(t+\sum_{j\in J_\pit(p-1)}t_{j+1}\bigg)\bigg)\,d\lambda\big(t_1,t_2,(t_{j+1})_{j\in J_\pit\setminus\{1\}}\big)
\end{aligned}\\
&=\bigg(\int_0^1g_1(r)\,dr\bigg) \\
&\quad\cdot\int_{E(\pit,t)}
\begin{aligned}[t]
&g_2(t+u_1)\bigg(\prod_{p=3}^{m-1}g_p(t+u_1+\sum_{j\in J_\pit(p-1)\setminus\{1\}}u_j\bigg)\bigg) \\
&\qquad\cdot\bigg(\prod_{p=m}^{n-1}g_p\bigg(t+\sum_{j\in J_\pit(p-1)}u_j\bigg)\bigg)\,d\lambda\big(u_1,(u_j)_{j\in J_\pit\setminus\{1\}}\big)
\end{aligned} \\
&=\tau(g_1)\Lambda_{\pit}(g_2,\ldots,g_{n-1}).
\end{align*}
This proves the result in the case $k=1$.

Suppose $k>1$.
For any $g_n\in C[0,1]$, we will show
\[
\tau\big(\Lambda_\pi(g_1,\ldots,g_{n-1})g_n\big)=\tau\big(\Lambda_{\pit}(g_1,g_2,\ldots,\widehat{g_k},\ldots,g_{n-1})g_n\big)\tau(g_k),
\]
which will finish the proof.
We will rotate and appeal to the case just proved.
Indeed, the partition obtained from $\ell_n^{k-1}(\pi)$ be gluing together $1$ and $2$ is just $\ell_{n-1}^{k-1}(\pit)$.
By Lemma~\ref{lem:rotate} and the case just proved, we have
\begin{align*}
\tau\big(\Lambda_\pi(g_1,\ldots,g_{n-1})g_n\big)&=\tau\big(\Lambda_{\ell_n^{k-1}(\pi)}(g_k,g_{k+1},\ldots,g_n,g_1,\ldots,g_{k-2})g_{k-1}\big) \\
&=\tau(g_k)\tau\big(\Lambda_{\ell_{n-1}^{k-1}(\pit)}(g_{k+1},\ldots,g_n,g_1,\ldots,g_{k-2})g_{k-1}\big) \\
&=\tau(g_k)\tau\big(\Lambda_{\pit}(g_1,\ldots,g_{k-1},g_{k+1},\ldots,g_{n-1})g_n\big),
\end{align*}
as required.
\end{proof}

\begin{lemma}\label{lem:1andn}
Let $n\ge2$ and suppose $\pi\in\Pc(n)$ has $1$ and $n$ in the same block.
Then for all $g_1,\ldots,g_{n-1}\in C[0,1]$, $\Lambda_\pi(g_1,g_2,\ldots,g_{n-1})$ is a constant function.
Moreover, 
letting $\pit\in\Pc(n-1)$ be the restriction of $\pi$ to $\{1,\ldots,n-1\}$, 
we have
\[
\Lambda_\pi(g_1,\ldots,g_{n-1})=\tau\big(\Lambda_\pit(g_1,\ldots,g_{n-2})g_{n-1}\big).
\]
\end{lemma}
\begin{proof}
It will suffice to show that, for every $g_n\in C[0,1]$, we have
\[
\tau\big(\Lambda_\pi(g_1,\ldots,g_{n-1})g_n\big)
=\tau\big(\Lambda_\pit(g_1,\ldots,g_{n-2})g_{n-1}\big)\tau(g_n).
\]
Let $\sigma\in\Pc(n)$ be obtained from $\pi$ by right rotating, so that $\pi=\ell_n(\sigma)$.
Then $\{1,2\}\subseteq S_\sigma(1)$ and $\pit$ equals the partition obtained from $\sigma$ by gluing together $1$ and $2$.
Using Lemma~\ref{lem:rotate} and Lemma~\ref{lem:adjacent}, we have
\begin{align*}
\tau\big(\Lambda_\pi(g_1,\ldots,g_{n-1})g_n\big)
&=\tau\big(\Lambda_{\ell_n(\sigma)}(g_1,\ldots,g_{n-1})g_n\big)
=\tau(\Lambda_\sigma(g_n,g_1,\ldots,g_{n-2})g_{n-1}\big) \\
&=\tau(g_n)\tau\big(\Lambda_\pit(g_1,\ldots,g_{n-2})g_{n-1}\big).
\end{align*}
\end{proof}

\medskip
The following is easily checked:
\begin{lemma}\label{lem:piGam}
Suppose $\pi$ is any partition of $\{1,\ldots,n\}$.
Then
\[
\Gamma_\pi(1,\ldots,1)=1.
\]
\end{lemma}
Here is an immediate consequence of the above fact and~\eqref{eq:EXX*}.
\begin{prop}\label{prop:scalarMoments}
For any $n\in\Nats$ and any $b_1,\ldots,b_n\in B$,
\[
\Ec(Xb_1X^*Xb_2X^*\cdots Xb_nX^*)\in\Cpx1.
\]
\end{prop}

From Lemmas \ref{lem:1n}-\ref{lem:1andn}, we easily get the following:
\begin{lemma}\label{lem:ncLam}
Suppose $\pi$ is a noncrossing partition of $\{1,\ldots,n\}$.
Then
\[
\Lambda_\pi(1,1,\ldots,1)=1.
\]
\end{lemma}

Since all partitions of $\{1,2,3\}$ are noncrossing, from \eqref{eq:EX*X}-\eqref{eq:EXX*} we easily get:
\begin{gather}
\Ec(XX^*)=\Ec(X^*X)=1, \label{eq:EXX*1} \\
\Ec\big((XX^*)^2\big)=\Ec\big((X^*X)^2\big)=2,  \label{eq:EXX*2} \\
\Ec\big((XX^*)^3\big)=\Ec\big((X^*X)^3\big)=5. \notag
\end{gather}

There are 14 noncrossing partitions of $\{1,2,3,4\}$ and one crossing partition, namely, $\pi_4=\{\{1,3\},\{2,4\}\}$.
We have
\[
E(\pi_4,t)=\{(t_1,t_2)\in\Reals^2\mid 0<t+t_1\le 1,\;0<t+t_1+t_2\le1,\;0<t+t_2\le1\}.
\]
Using the definition~\eqref{eq:Lambdapi} of $\Lambda_{\pi}$ and making a change of variables, we calculate, for $g_1,g_2,g_3\in C[0,1]$,
\begin{multline*}
\Lambda_{\pi_4}(g_1,g_2,g_3)(t)
=\int_{-t}^{1-t}g_1(t+t_1)\int_{\max(-t,-t-t_1)}^{\min(1-t,1-t-t_1)}g_2(t+t_1+t_2)g_3(t+t_2)\,dt_2\,dt_1 \\
=\int_0^tg_1(x)\int_0^{1-t+x}g_2(y)g_3(t-x+y)\,dy\,dx
+\int_t^1g_1(x)\int_{x-t}^1g_2(y)g_3(t-x+y)\,dy\,dx.
\end{multline*}
From this,
we calculate
\[
\Lambda_{\pi_4}(1,1,1)(t)=\frac12+t(1-t).
\]
Consequently, from~\eqref{eq:EX*X} and~\eqref{eq:EXX*}, we find
\begin{align}
\Ec\big((XX^*)^4\big)&=14+\frac23 \label{eq:EXX*4} \\[1ex]
\Ec\big((X^*X)^4\big)(t)&=14+\frac12+t(1-t). \label{eq:EX*X4}
\end{align}

Unlike with scalar-valued R-diagonality in the tracial setting,
in the $B$-valued case, $*$-freeness is not guaranteed in a polar decomposition.
This phenomenon was seen in~\cite{BD:rdiag}, but is also exhibited by the asymptotic limit of the random Vandermonde matrices:
\begin{prop}\label{prop:polarNotFree}
The element $X$ does not have the same $*$-distribution as any element in a $B$-valued
$*$-noncommutative probability space of the form $PU$, with $U$ unitary, $P\ge0$ and such that $U$ and $P$ are $*$-free
over $B$.
\end{prop}
\begin{proof}
Suppose for contradiction such a realization $X\sim PU$ is possible
for $P$ and $U$ in a $B$-valued $*$-noncommutative probability space $(\At,\Ect)$.

From \eqref{eq:EXX*4}-\eqref{eq:EX*X4},
we have
\begin{align*}
\Ect(P^8)&=14+\frac23 \\[1ex]
\Ect(U^*P^8U)(t)&=14+\frac12+t(1-t).
\end{align*}
However, by $*$-freeness, we calculate
\[
\Ect(U^*P^8U)=\Ect\big(U^*\Ect(P^8)U\big)=\Ect\big(U^*(14+\frac23)U\big)=14+\frac23,
\]
which is a contradiction.
\end{proof}

\begin{ques}
Can $X$ have the same $*$-distribution as a product $UP$
for some $U$ and $P$ as described in Proposition~\ref{prop:polarNotFree}?
\end{ques}

The next result answers negatively a question of G.\ Tucci.
\begin{prop}
With respect to the trace $\tau\circ\Ec$, $X$ is not a scalar-valued R-diagonal element.
\end{prop}
\begin{proof}
If it were scalar-valued R-diagonal, then,
because $XX^*$ and $X^*X$ would be free with respect to $\tau\circ\Ec$,
we would have
\begin{equation}\label{eq:tauEXXpoly}
\tau\circ\Ec\big(\big((X^*X)^4-\frac{44}3\big)\big((XX^*)^2-2\big)\big((X^*X)^4-\frac{44}3\big)((XX^*)^2-2\big)\big)=0.
\end{equation}
Letting $b\in B$
be $b(t)=14+\frac12+t(1-t)$, by $B$-valued R-diagonality of $X$ and~\eqref{eq:EX*X4} and~\eqref{eq:EXX*2}, we have
\begin{equation}\label{eq:EXXpoly}
\Ec\big(\big((X^*X)^4-b\big)\big((XX^*)^2-2\big)\big((X^*X)^4-b\big)\big((XX^*)^2-2\big)\big)=0.
\end{equation}
Writing $(X^*X)^4-\frac{44}3=((X^*X)^4-b)+(b-\frac{44}3)$, expanding, distributing, using~\eqref{eq:EXXpoly},
$B$-valued R-diagonality again and~\eqref{eq:EXX*2}, we get
\begin{align}
\Ec\big(\big((X^*X)^4-\frac{44}3\big)&\big((XX^*)^2-2\big)\big((X^*X)^4-\frac{44}3\big)\big((XX^*)^2-2\big)\big) \notag \\
&=\Ec\big((b-\frac{44}3)\big((XX^*)^2-2\big)(b-\frac{44}3)\big((XX^*)^2-2\big)\big) \notag \\
&=(b-\frac{44}3)\left(\Ec\big((XX^*)^2(b-\frac{44}3)(XX^*)^2\big)-4(b-\frac{44}3)\right). \label{eq:Eprod}
\end{align}
Using~\eqref{eq:EXX*} we find that for $b'\in B$,
\[
\Ec\big((XX^*)^2b'(XX^*)^2\big)
=\sum_{\pi\in\Pc(4)}\Gamma_\pi(1,1,b',1)\tau\big(\Lambda_\pi(1,1,1)\big)
=10\tau(b')+(4+\frac23)b'
\]
and, thus, that the quantity~\eqref{eq:Eprod} equals $\frac23(b-\frac{44}3)^2$.
But
\[
\frac23\tau\big((b-\frac{44}3)^2\big)=\frac23\int_0^1((t(1-t)-\frac16)^2\,dt=\frac1{270}\ne0,
\]
which shows that~\eqref{eq:tauEXXpoly} fails to hold.
\end{proof}

\medskip
We conclude this paper with a report of calculations of some of the 
$C[0,1]$-valued cumulant maps of the asymptotic $*$-distribution
of random Vandermonde matrices, namely, of the $C[0,1]$-valued distribution $\Ec$ from Theorem~\ref{thm:XRDiag}.
The details of these calculations are either straightforward to work out or can be found in the Mathematica~\cite{W15}
Notebook that is available with this paper.
Let $\alpha$ denote these cumulant maps, and for brevity let
\[
\alpha^{(1)}_k=\alpha_{(\underset{2k}{\underbrace{\scriptstyle 1,2,\ldots,1,2}})},\qquad\alpha^{(2)}_k=\alpha_{(\underset{2k}{\underbrace{\scriptstyle 2,1,\ldots,2,1}})}
\]
be those that need not, by virtue of R-diagonality, be zero.

We will use the following notion.
\begin{defi}\label{defi:PC}
For $n\in\Nats$, a partition $\pi\in\Pc(n)$ is said to be {\em purely crossing} if
\begin{enumerate}[label=(\alph*),leftmargin=20pt]
\item\label{it:nosplit} no proper subinterval $\{p+1,p+2,\ldots,p+q\}$ of $\{1,\ldots,n\}$ splits $\pi$
(by proper subinterval we mean with $0\le p<p+q\le n$ and $q<n$)
\item\label{it:noneigh} no block of $\pi$ contains neighbors (modulo $n$), namely, $k\overset{\pi}{\not\sim}k+1$
for all $k\in\{1,\ldots,n-1\}$ and  $1\overset{\pi}{\not\sim}n$.
\end{enumerate}
We let $\PC(n)$ denote the set of all purely crossing partitions of $\{1,\ldots,n\}$.
\end{defi}
Note that condition~\ref{it:nosplit} implies that $\pi$ has no singleton blocks.
It is easy to check that $\PC(n)$ is empty when $n\in\{1,2,3,5\}$,
and that $\PC(4)=\{\pi_4\}$, where $\pi_4=\{\{1,3\},\{2,4\}\}$.
The purely crossing projections and related quantities are studied further in~\cite{D}.

\begin{prop}
We have
\[
\alpha^{(1)}_1(b_1)=\alpha^{(2)}_1(b_1)=\tau(b_1)1.
\]
For $n\in\{2,3,4,5,6,7\}$, we have
\begin{align*}
\alpha^{(1)}_n(b_1,\ldots,b_{2n-1})&=\sum_{\pi\in\PC(n)}\Gamma_{\pi}(1,b_2,b_4,\ldots,b_{2n-2})
\tau\big(\Lambda_{\pi}(b_1,b_3,\ldots,b_{2n-3})b_{2n-1}\big), \\
\alpha^{(2)}_n(b_1,\ldots,b_{2n-1})&=\sum_{\pi\in\PC(n)}\tau\big(\Gamma_{\pi}(b_1,b_3,\ldots,b_{2n-1})\big)\Lambda_{\pi}(b_2,b_4,\ldots,b_{2n-2}).
\end{align*}
In particular, $\alpha^{(1)}_n$ and $\alpha^{(2)}_n$ vanish when $n\in\{2,3,5\}$.
However, this pattern breaks with $n=8$, for we have
\begin{align*}
\alpha^{(1)}_8(b_1,\ldots,b_{15})
&=\sum_{\pi\in\PC(8)}\Gamma_{\pi}(1,b_2,b_4,\ldots,b_{14})
\tau\big(\Lambda_{\pi}(b_1,b_3,\ldots,b_{13})b_{15}\big) \\
&\hspace{2em}
- \tau(b_{2} b_{6} ) b_{4} b_{8} \tau(b_{10} b_{14} ) b_{12}
  \tau(\Lambda_{\pi_4}(b_{1} , b_{3} , b_{5} )  b_{7} )
  \tau(\Lambda_{\pi_4}(b_{9} , b_{11} , b_{13} )  b_{15}) \\
&\hspace{2em}
- \tau(b_{2} b_{6} b_{10} b_{14} ) b_{4} \tau(b_{8} b_{12} )
  \tau( \Lambda_{\pi_4}(b_{1} , b_{3} , b_{5} ) b_{15} )
  \tau( \Lambda_{\pi_4}(b_{7} , b_{9} , b_{11} )  b_{13} )\\
&\hspace{2em}
- \tau( b_{2} b_{14} ) b_{4} b_{8}  \tau(b_{6} b_{10} ) b_{12}
   \tau(\Lambda_{\pi_4}(b_{1} , b_{3} , b_{13} )   b_{15} )
   \tau(\Lambda_{\pi_4}(b_{5} , b_{7} , b_{9} )  b_{11} ) \\
&\hspace{2em}
-\tau(b_{2} b_{6} b_{10} b_{14} ) \tau(b_{4} b_{8} ) b_{12} 
  \tau(\Lambda_{\pi_4}(b_{1} , b_{11} , b_{13} )   b_{15} )
  \tau( \Lambda_{\pi_4}(b_{3} , b_{5} , b_{7} )   b_{9} ), \\
\alpha^{(2)}_8(b_1,\ldots,b_{15})
&=\sum_{\pi\in\PC(8)}\tau\big(\Gamma_{\pi}(b_1,b_3,\ldots,b_{15})\big)\Lambda_{\pi}(b_2,b_4,\ldots,b_{14}) \\
&\hspace{2em}
- \tau(b_{1} b_{5} b_{9} b_{13} )  \tau( b_{3} b_{7}) \tau(b_{11} b_{15})
   \tau(\Lambda_{\pi_4}(b_{2}, b_{4}, b_{6}) b_{8} ) \Lambda_{\pi_4}(b_{10}, b_{12}, b_{14}) \\
&\hspace{2em}
- \tau(b_{1} b_{13}) \tau( b_{3} b_{7} b_{11} b_{15}) \tau(b_{5} b_{9})
   \Lambda_{\pi_4}(b_{2}, b_{12}, b_{14}) \tau(\Lambda_{\pi_4}(b_{4}, b_{6}, b_{8}) b_{10}) \\
&\hspace{2em}
- \tau(b_{1} b_{5} b_{9} b_{13}) \tau(b_{3}b_{15} ) \tau(  b_{7}b_{11})
  \Lambda_{\pi_4}(b_{2}, b_{4}, b_{14}) \tau(\Lambda_{\pi_4}(b_{6}, b_{8}, b_{10}) b_{12} ) \\
&\hspace{2em}
-\tau(b_{1} b_{5}) \tau(b_{3} b_{7} b_{11} b_{15}) \tau(b_{9}b_{13})
  \Lambda_{\pi_4}(b_{2}, b_{4}, b_{6}) \tau(\Lambda_{\pi_4}(b_{8}, b_{10}, b_{12}) b_{14} ).
\end{align*}
\end{prop}

\end{document}